\documentclass[aap]{imsart}

\RequirePackage[OT1]{fontenc}
\RequirePackage{amsthm,amsmath,amsfonts,amssymb}
\RequirePackage[numbers]{natbib}
\RequirePackage[colorlinks,citecolor=blue,urlcolor=blue]{hyperref}
\usepackage[margin=3cm]{geometry}
\usepackage{mathtools}
\usepackage{lmodern}
\usepackage{graphicx}
\usepackage{color}
\usepackage{dsfont}
\usepackage{bbm}

\arxiv{}

\newtheorem{theorem}{Theorem}[section]
\newtheorem{lemma}[theorem]{Lemma}
\newtheorem{proposition}[theorem]{Proposition}
\newtheorem{corollary}[theorem]{Corollary}

\theoremstyle{remark}
\newtheorem{remark}[theorem]{Remark}
\newtheorem{example}[theorem]{Example}
\newtheorem{definition}[theorem]{Definition}

\newtheorem{assumption}[theorem]{Assumption}

\startlocaldefs

\def\be{\begin{eqnarray}}
\def\ee{\end{eqnarray}}
\def\b*{\begin{eqnarray*}}
\def\e*{\end{eqnarray*}}

\newcommand{\rmi}{{\rm (i)$\>\>$}}
\newcommand{\rmii}{{\rm (ii)$\>\>$}}
\newcommand{\rmiii}{{\rm (iii)$\>\>$}}

\def \E{\mathbb{E}}
\def \F{\mathbb{F}}

\def \L{\mathbb{L}}

\def \P{\mathbb{P}}

\def \R{\mathbb{R}}

\def \Z{\mathbb{Z}}

\def\Bc{{\cal B}}
\def\Cc{{\cal C}}
\def\Dc{{\cal D}}
\def\Ec{{\cal E}}
\def\Fc{{\cal F}}

\def\Hc{{\cal H}}

\def\Lc{{\cal L}}
\def\Mc{{\cal M}}
\def\Nc{{\cal N}}

\def\Pc{{\cal P}}
\def\Qc{{\cal Q}}

\def\Wc{{\cal W}}
\def\Xc{{\cal X}}
\def\Yc{{\cal Y}}

\def \eps {\varepsilon}

\def \m {\mu}
\def \n {\nu}

\def \vp {\varphi}
\def \ps {\psi}

\def\q{\quad}

\def \Lam {\Lambda}

\def \Om {\Omega}

\def \sb {\subseteq}

\def \pq {\preceq}

\def \lro {\longrightarrow}

\def \pp {\mathsf{P}}
\def \dd {\mathsf{D}}
\def \no{\noindent}
\def \supp{{\rm{supp}}}

\def \bm{\boldsymbol{\mu}}
\def \bn{\boldsymbol{\nu}}

\def \bS{\mathsf{S}}

\def \bx{\mathsf{x}}
\def \by{\mathsf{y}}
\def \bz{\mathsf{z}}

\def \bq{\mathsf{q}}
\def \bB{\mathsf{B}}

\def \li{\rm Lip}
\def \x{\times}
\def \ox{\otimes}
\def \op{\oplus}
\def \1{\mathbf 1}

\def \mn{{\mu}^{(n)}}

\def \tmn{\tilde{\mu}^{(n)}_m}
\def \tn{\tilde{\mu}^{(n)}}
\def \tm{\tilde{\nu}^{(n)}}
\def \tmkn{\tilde{\mu}^{(n)}_{k,m}}

\def \hm{\hat{\mu}^{n}}
\def \hn{\hat{\nu}^{n}}
\def \hnm{\hat{n}_{m}}

\def \hbm{\hat{\bm}^{n}}

\def \0{\mathsf 0}
\def \1{\mathsf 1}

\endlocaldefs

\begin{document}

\begin{frontmatter}
\title{Computational Methods for Martingale Optimal\\ Transport Problems\thanksref{T1}}
\runtitle{Computational Methods for MOT Problems}
\thankstext{T1}{This research is supported by the European Research Council under the European Union's Seventh Framework Programme (FP7/2007-2013) / ERC grant agreement no. 335421. The authors thank Guillaume Carlier, Bruno L\'evy, Tongseok Lim, Terry Lyons and Pietro Siorpaes for insightful discussions and comments.}

\begin{aug}
\author{\fnms{Gaoyue} \snm{Guo}\ead[label=e1]{gaoyue.guo@maths.ox.ac.uk}\ead[label=u1,url]{http://www.maths.ox.ac.uk/people/gaoyue.guo}}
\and
\author{\fnms{Jan} \snm{Ob{\l}{\'o}j}\thanksref{t2}\ead[label=e2]{jan.obloj@maths.ox.ac.uk}\ead[label=u2,url]{http://www.maths.ox.ac.uk/people/jan.obloj}}

\thankstext{t2}{The second author also gratefully acknowledges support from St John's College, Oxford.}
\runauthor{Computational Methods for MOT Problems}

\affiliation{University of Oxford, United Kingdom}


\address{Mathematical Institute \\
University of Oxford \\
AWB, ROQ, Woodstock Road \\
Oxford OX2 6GG \\
United Kingdom\\
\printead{e1}\\
\phantom{E-mail:\ }\printead*{e2}}
\end{aug}

\maketitle 

\begin{abstract}
We develop computational methods for solving the {\em martingale optimal transport} (\textsf{MOT}) problem --- a version of the classical optimal transport with an additional martingale constraint on the transport's dynamics. We prove that a general, multi-step multi-dimensional, \textsf{MOT} problem can be approximated through a sequence of {\em linear programming} (\textsf{LP}) problems which result from a discretization of the marginal distributions combined with an appropriate relaxation of the martingale condition. We further furnish two generic approaches for discretizing probability distributions,  suitable respectively for the cases when we can compute integrals against these distributions or when we can sample from them. These render our main result applicable and lead to an implementable numerical scheme for solving \textsf{MOT} problems. Finally, specializing to the one-step model on real line, we provide an estimate of the convergence rate which, to the best of our knowledge, is the first of its kind in the literature.
\end{abstract}

\begin{keyword}[class=MSC]
\kwd[Primary ]{49M25}
\kwd{60H99}
\kwd[; secondary ]{90C08.}
\end{keyword}

\begin{keyword}
\kwd{martingale optimal transport}
\kwd{martingale relaxation}
\kwd{robust hedging}
\kwd{duality}
\kwd{discretization of measure}
\kwd{linear programming}
\end{keyword}

\end{frontmatter}

\section{Introduction}\label{sec:intro}

The {\em optimal transport}  (\textsf{OT}) problem is concerned with transferring mass from one location to another in such a way as to optimize a given criterion. 
Rephrased mathematically, and for simplicity considering the one-dimensional case, we are given two probability distributions $\m$ and $\n$ on $\R$ and seek to minimize 
\be\label{intergral}
\int_{\R^2} c(x,y) \P(dx,dy),
\ee 
among all probability measures $\P$, also known as {\em transport plans}, such that 
\be
\label{def:ot_plan}
\P\big[E\x \R\big]~=~\m[E] & \mbox{and} & \P\big[\R\x E\big]~=~\n[E],~ \mbox{ for all } E\in\Bc(\R),
\ee
where $c:\R^2\to\R$ is a measurable cost function. Theoretical advances in the last fifty years  characterize existence, uniqueness, representation and smoothness properties of optimizers in a variety of different settings, see e.g. \cite{RR, Villani}, and applications are abundant throughout most of the applied sciences  including biomedical sciences, geography and data science. Accordingly, numerical techniques for the \textsf{OT} are of great importance and have rapidly developed into an important and separate field of applied mathematics:

\begin{enumerate}
\item In the absolutely continuous case, i.e. $\m(dx)=\rho(x)dx$ and $\n(dy)=\sigma(y)dy$, Benamou and Brenier proposed in \cite{BB} a numerical scheme for the quadratic distance function $c(x,y)=(x-y)^2$  using an equivalent formulation arising from fluid mechanics.
\item In the purely discrete case, i.e. $\m(dx)=\sum_{i=1}^m\alpha_i\delta_{x_i}(dx)$ and $\n(dy)=\sum_{j=1}^n \beta_j\delta_{y_j}(dy)$,  the \textsf{OT} problem reduces to a {\em linear programming} (\textsf{LP}) problem  and can be computed using the iterative Bregman projection, see Benamou et al. \cite{BCCNP}.
\item In the semi-discrete case, i.e.  $\m(dx)=\rho(x)dx$ and $\n(dy)=\sum_{j=1}^n \beta_j\delta_{y_j}(dy)$, L\'evy \cite{Levy} adopted a computational geometry approach to the cost $c(x,y)=(x-y)^2$ and solved the \textsf{OT} problem by means of Laguerre's  tessellations.
\end{enumerate}

Recently, an additional constraint has been taken into account, which leads to the so-called {\em martingale optimal transport} (\textsf{MOT}) problem. This optimization problem was motivated by, and contributed to, the so-called model-independent, or robust, pricing of exotic options in mathematical finance, perspective which has gained significant momentum in the wake of financial crisis. More precisely, the two given measures $\m$ and $\n$  describe the initial and final distributions of stock prices and can be recovered from market prices of traded call/put  options. Calibrated market models are thus identified by martingales with these prescribed marginals, i.e. transport plans $\P$ which further satisfy 
\be\label{def:martingale}
\int_{\R} y \P_{x}(dy) ~=~ x,~  \mbox{ for } \m\mbox{ - a.e. } x\in\R,
\ee
where $(\P_x)_{x\in\R}$ denotes the disintegration of $\P$ with respect to $\m$. The \textsf{MOT} problem aims at  maximizing\footnote{The maximization formulation is more adapted to financial applications. We refer to $c$ as a reward function or payoff, which is commonly accepted in the finance jargon.} the integral \eqref{intergral} overall $\P$, still named transport plans, satisfying  the constraints \eqref{def:ot_plan} and  \eqref{def:martingale}, and it corresponds to the model-independent price for option $c$. This methodology was presented by Beiglb\"ock et al. \cite{BHLP}, to which we refer for a more detailed discussion. It is also worth mentioning that some concrete \textsf{MOT} problems for particular payoffs have been investigated, by means of stochastic control or Skorokhod embedding techniques, in a stream of papers going back to Hobson \cite{Hobson}, see e.g. \cite{BT, CGHL, GHLT, BHR, CO2, CO1, HK2, HN, Guo}. 

Given the active theoretical interest in \textsf{MOT}  problems, as well as their importance for applications in mathematical finance, it becomes increasingly important to develop numerical techniques and computational methods for these problems. A natural starting point is given by a simple, but important, observation that, for the purely discrete case stated above, the \textsf{MOT} problem is equivalent to the following \textsf{LP} problem:
\begin{align*}
\max_{(p_{i,j})_{1\le i\le m,1\le j\le n}\in\R_+^{mn}}~ \sum_{i=1}^m \sum_{j=1}^n  p_{i,j}c(x_i, y_j) 
 \q\mbox{s.t.}\q & \sum_{j=1}^n p_{i,j} ~=~ \alpha_i,~  \mbox{ for } i=1,\ldots, m,\\
& \sum_{i=1}^m p_{i,j} ~=~\beta_j,~ \mbox{ for } j=1,\ldots, n,\\
  & \sum_{j=1}^n p_{i,j}y_j~=~ \alpha_ix_i,~ \mbox{ for } i=1,\ldots, m.
\end{align*}
Such \textsf{LP} formulation was pioneered in Davis et al. \cite{DOR}, where instead of the marginal constraint $\nu$, only a finite number of expectation constraints are given. This, for a convex reward function, leads to optimizers with finite support. To adapt this approach in general, we could hope to approximate the \textsf{MOT} problem for $(\m,\n)$ with the above \textsf{LP} problem for finitely supported $(\m^n,\n^n)$ which are `close' to $(\m,\n)$. Unfortunately, this naive idea hits two important obstacles. First, there are no general continuity results of the \textsf{MOT} problem as a function of its input $(\m,\n)$. To the best of our knowledge, the only exception is Juillet \cite{Juillet} who proved that, if $c(x,y)=\vp(x)\ps(y)$ or $c(x,y)=h(x-y)$, where $\vp$, $\ps$, $h:\R\to\R$ are assumed to satisfy the conditions of Remark 2.10 in \cite{Juillet}, then there exists an optimizer $\P^*(\m,\n)$ which is Lipschitz with respect to $(\m,\n)$ under a topology of Wasserstein type. We extend his result to more general payoffs $c$ in Proposition \ref{prop:conti} but it remains a one-dimensional result. Second, even if $(\m,\n)$ admits a martingale transport plan, in dimensions $d>1$ it may be very difficult to construct a discrete approximation $(\m^n,\n^n)$ which also does so, see Remark \ref{rk:discretization_convexorder} below.
In fact, the martingale condition, which seems harmless, renders any of the usual \textsf{OT} techniques unusable, e.g. stability results, tools from PDE and computational geometry. 
To the best of our knowledge, and in contrast to the \textsf{OT}, numerical methods for \textsf{MOT} problems are close to non-existent so far, relative to the theory and applications.  

This paper fills in this important gap. We provide an approximation  approach for solving systematically  $N-$period  \textsf{MOT} problems on $\R^d$, with $N\ge 2$ and $d\ge 1$. Our approximation of the original problem relies on a  discretization of the marginal distributions coupled with a suitable relaxation of the martingale constraint leading to a sequence of \textsf{LP}  problems. This sequence converges and, specializing to  $N=2$ and $d=1$, we obtain the convergence speed. Our investigation involves a number of novel results and techniques which, we believe, are of independent interest. In particular, we compute explicitly the constants in \cite{FG} for the convergence rate of empirical measures to the limit in Glivenko-Cantelli's theorem.

The paper is organized as follows. 
In the rest of this introduction, we clarify the framework and notations  under which we work. Section \ref{sec:result} contains all the main theoretical results: we introduce the {\em relaxed  martingale optimal transport} (relaxed \textsf{MOT}), show the convergence of approximating \textsf{LP}  problems to the \textsf{MOT} problem and provide a bound on the convergence rate in dimension one. In Section \ref{sec:num_pb}, we consider possible implementations of our method. This requires approximating a probability measure $\mu$ by  discrete measures $\mu^n$ and being able to compute, or bound, the Wasserstein distance between $\mu^n$ and $\mu$. We develop two generic approaches to achieve this, and then present several numerical examples which illustrate our methods and provide heuristic insights into the structure of optimizers, including a conjecture in \cite{LGK}. Section \ref{sec:proof} contains all the related proofs. Section \ref{sec:extension} concludes the paper and points to potential future work.

\subsection{Preliminaries}

For a given set $E$, we denote by $E^k$ its $k-$fold product. If $E$ is Polish, then $\Bc(E)$ denotes its Borel $\sigma-$field and $\Pc(E)$ is the set of probability measures on $\big(E, \Bc(E)\big)$ which admit a finite first moment.
As is common when studying the \textsf{OT},  we formulate our problem on the canonical space, which plays an important role in the analysis.  Let $\Om:=\R^d$ with its elements denoted by $\bx=(x_1,\ldots, x_d)$ and $\Pc:= \Pc(\Om)$. Throughout, we endow $\R^d$ with the $\ell_1$ norm   $|\cdot|$, i.e.  $|\bx| :=\sum_{i=1}^d |x_i|$.
Define $\Lam$ to be the space of Lipschitz functions on $\R^d$ and, given $f\in\Lam$, denote by ${\li}(f)$ its Lipschitz constant on $\R^d$. For each $L>0$, let $\Lam_L\subset\Lam$ be the subspace of functions $f$ with ${\li}(f)\le L$. We consider the coordinate process $(\bS_k)_{1\le k\le N}$, i.e. $\bS_k(\bx_1,\ldots, \bx_N):=\bx_k$ for all $(\bx_1,\ldots, \bx_N)\in\Om^N$, and its natural filtration  $(\Fc_k)_{1\le k\le N}$, i.e. $\Fc_k:=\sigma(\bS_1,\ldots, \bS_k)$. From the financial point of view, $\Om^N$ models the collection of all possible trajectories for the price evolution of $d$ stocks, where $N$ is the number of trading dates. 

Given a vector of probability measures $\bm=(\m_k)_{1\le k\le N}\in\Pc^N$, define the set of transport plans with the marginal distributions $\m_1, \ldots, \m_N$ by
\b*
\Pc(\bm) ~:=~ \Big\{\P\in\Pc\big(\Om^N\big):\q   \P \circ \bS_k^{-1} ~=~ \m_k,~ \mbox{ for } k=1,\ldots, N\Big\},
\e* 
where $\P\circ \bS_k^{-1}$ denotes the  push forward of $\P$ via the map $\bS_k: \Om^N\to\Om$. In particular, the Wasserstein distance (of order $1$) between $\m$ and $\n\in\Pc$ is given by 
\be \label{def:wass_dual}
~~~~~~~~~ \Wc(\mu,\nu) ~~:=~~ \inf_{\P\in\Pc(\m,\n)}~   \E_{\P}\big[\big|\bS_1-\bS_2\big|\big] ~~ = ~~ \sup_{f\in\Lam_1}~  \left\{\int_{\R^d} f(\bx)\m(d\bx) - \int_{\R^d} f(\bx)\n(d\bx) \right\},
\ee 
where the second equality follows by Kantorovich's duality. We recall that $\Pc$, equipped with the metric $\Wc$, is a Polish space. Further, for any $(\mu^n)_{n\ge 1}\subset\Pc$ and $\mu\in\Pc$, $\Wc(\mu^n, \mu)\to 0$ holds if and only if 
\b*
\mu^n~ \stackrel{\Lc}{\lro}~ \mu & \mbox{and} & 
 \int_{\R^d} |\bx|\mu^n(d\bx) ~\longrightarrow~ \int_{\R^d} |\bx|\mu(d\bx),
\e*
where $\Lc$ represents the weak convergence of probability measures, see the monograph of Rachev and R\"uschendorf \cite{RR} for more details. To facilitate  our analysis in the sequel, we endow $\Pc^N$ with the product metric $\Wc^{\oplus}$ defined by  $
\Wc^{\oplus}(\bm,\bn) := \sum_{k=1}^N  \Wc (\mu_k,\n_k)$, for all $\bm$, $\bn \in \Pc^N$.
It follows that $\Pc^N$ is Polish with respect to $\Wc^{\oplus}$. We close this introduction by listing some notations used in the following.

\vspace{3mm}

\no {\bf Notations.} 

\begin{itemize}
\item $\0:=(0,\ldots,0)$, $\1:=(1,\ldots, 1)\in \R^d$, and to stress the unidimensional case, we write $\bx\equiv x$ and $\bS_k\equiv S_k$ for $d=1$, see e.g. Section \ref{ssec:example}. 

\vspace{1mm}

\item $\L^0(\Om^k;\R^d)$ is the set of measurable functions from $\Om^k$ to $\R^d$. Denote by $\L^{\infty}(\Om^k;\R^d)\subset \L^0(\Om^k;\R^d)$ the subset of (uniformly) bounded functions, and by $\Cc_{b}(\Om^k;\R^d)\subset \L^{\infty}(\Om^k;\R^d)$ the subset of continuous bounded functions.

\vspace{1mm}

\item For simplicity purposes, we adopt the  abbreviations below whenever the context is clear:
\b*
 \int f d\m \equiv \int_{\R^d} f(\bx) \m(d\bx),\q (p_{i_1,\ldots, i_N})\equiv (p_{i_1,\ldots, i_N})_{i_1\in I_1,\ldots, i_N\in I_N},\q \sum_{i_1,\ldots, i_N}\equiv \sum_{i_1\in I_1,\ldots, i_N\in I_N}. 
\e*
\end{itemize}

\section{Main results}\label{sec:result}

Our computational method relies on the convergence result stated in Theorem \ref{thm:general}. To introduce the result, we need the notion of {\em$\eps-$approximating martingale measure}.

\begin{definition}\label{def:rmm}
For any $\eps\ge 0$, a probability measure $\P\in\Pc\big(\Om^N\big)$ is said to be an $\eps-$approximating martingale measure if for each  $k=1,\ldots, N-1$
\be\label{def:rtp2}
\E_{\P}\Big[\Big|\E_{\P}\big[\bS_{k+1}\big|\Fc_{k}\big]~-~\bS_k\Big|\Big] ~\le ~ \eps,
\ee
or equivalently, in view of the monotone class theorem,   
\be\label{def:rtp}
\E_{\P}\big[h(\bS_{1},\ldots, \bS_{k})\cdot (\bS_{k+1}-\bS_{k})\big] ~\le ~ \eps\|h\|_{\infty},~ \mbox{ for all } h\in\Cc_b(\Om^k;\R^d),
\ee
where $\|h\|_{\infty}:=\max\big(\|h^{(1)}\|_{\infty},\ldots, \|h^{(d)}\|_{\infty}\big)$ and $\|h^{(i)}\|_{\infty}:=\sup_{(\bx_1,\ldots, \bx_k)\in\Om^k} \big|h^{(i)}(\bx_1,\ldots, \bx_k)\big|$ for $i=1,\ldots, d$.
\end{definition}

Given $\eps\ge 0$,  let $\Mc_{\eps}(\bm)\subset\Pc(\bm)$ be the subset containing all $\eps-$approximating martingale measures. Then $\Mc_{\eps}(\bm)$ is convex and closed with respect to the weak topology by \eqref{def:rtp}, and thus compact. For a measurable function $c:\Om^N\to\R$, the relaxed \textsf{MOT} problem is  defined by 
\be\label{def:rmotp}
\pp_{\eps}(\bm) ~:=~ \sup_{\P\in \Mc_{\eps}(\bm)}~ \E_{\P}\big[ c(\bS_1,\ldots, \bS_N)\big],
\ee
where we set by convention $\pp_{\eps}(\bm) :=-\infty$ whenever $\Mc_{\eps}(\bm)=\emptyset$. Denote further by $\Pc^{\pq}_{\eps}\subset\Pc^N$ the collection of measures $\bm$ such that $\Mc_{\eps}(\bm)\neq\emptyset$. We note that every $\P\in \Mc_{0}(\bm)$ is a martingale measure, i.e. $(\bS_k)_{1\le k\le N}$ is a martingale under $\P$, and $\pp_0(\bm)$ is the \textsf{MOT} problem. In the rest of the paper, for simplicity, we drop the subscript $\eps$ when $\eps=0$, e.g. $\Pc^{\pq} \equiv \Pc_{0}^{\pq}$, $\Mc(\bm)\equiv \Mc_{0}(\bm)$, $\pp(\bm) \equiv \pp_{0}(\bm)$, etc. 

As previously mentioned, $\pp(\bm)$ reduces to an \textsf{LP} problem once the marginals $\m_k$ have finite support for $k=1,\ldots, N$. We now couple this observation with a suitable relaxation of the martingale constraint to obtain a unified framework for computing $\pp(\bm)$ numerically. 

\begin{theorem}\label{thm:general}
Fix $\bm\in\Pc^{\pq}$. Let $(\bm^n)_{n\ge 1}\subset\Pc^N$ be a sequence converging to $\bm$ under $\Wc^{\oplus}$. Then, for all $n\geq 1$, $\bm^n\in \Pc^{\pq}_{r_n}$ with $r_n:=\Wc^{\op}(\bm^n,\bm)$. Assume further $c$ is Lipschitz. 

\vspace{1mm}

\no\rmi  For any sequence $(\eps_n)_{n\ge 1}$  converging to zero such that $\eps_n\ge r_n$ for all $n\ge 1$, one has
\b*
\lim_{n\to\infty} \pp_{\eps_n}(\bm^n) ~=~ \pp(\bm).
\e*
\rmii  For each $n\ge 1$, $\pp_{\eps_n}(\bm^n)$ admits an optimizer $\P_n\in\Mc_{\eps_n}(\bm^n)$, i.e. $\pp_{\eps_n}(\bm^n)=\E_{\P_n}[c]$. The sequence $(\P_n)_{n\ge 1}$ is tight and every limit point must be an optimizer for $\pp(\bm)$. In particular,  $(\P_n)_{n\ge 1}$ converges weakly whenever $\pp(\bm)$ has a unique optimizer.
\end{theorem}

\begin{remark}\label{rem:lp}
\rmi By Strassen's theorem \cite{Strassen}, $\bm\in \Pc^{\preceq}$ if and only if $\mu_k\pq \m_{k+1}$ for $k=1,\ldots, N-1$, or namely, $\int f d\m_k \le  \int  f d\m_{k+1}$ holds for all convex functions $f\in\Lam$ and $k=1,\ldots, N-1$. In addition, it follows by definition that $\Pc_{\eps}^{\preceq}\subset\Pc^N$ is convex and closed under $\Wc^{\oplus}$, and $\Mc(\bm)\subset\Mc_{\eps}(\bm)$ for all $\eps\ge 0$.

\vspace{1mm}

\no \rmii As noted before, a natural idea is to try to approximate $\pp(\bm)$ by $\pp(\bm^n)$ with finitely supported measures $\m_1^n,\ldots, \m_N^n$ since the latter amounts to an \textsf{LP} problem. For the classical \textsf{OT}, the continuous dependency of the optimization problem on $\bm$ can be derived either from the primal problem, or from its dual formulation. However, the additional martingale constraint means the usual \textsf{OT} arguments no longer work. The continuity of $\bm\mapsto \pp(\bm)$ remains an open question in general. For $d=1$, a partial result is shown in 
\cite{Juillet} and we extend it in Proposition \ref{prop:conti} below. Additionally, one has to consider suitable approximations, see Section \ref{ssec:convergence}, to even ensure that $\Mc(\bm^n)$ is nonempty. This becomes involved for $d>1$. Theorem \ref{thm:general} shows that, a  further relaxation of the martingale constraint allows to avoid both issues and to establish the desired convergence result. We also remark that the distance $r_n$ does not admit a closed-form expression and its  numerical estimation could be costly. Thanks to Theorem \ref{thm:general}, we may use in practice any upper bound $\eps_n\ge r_n$ converging to zero. 

\vspace{1mm}

\no \rmiii Finally, we point out the Lipschitz assumption can be slightly weakened. Let $E\sb\R^d$ be a closed subset such that $\supp(\m_k^n)\sb E$ for all $n\ge 1$ and $k=1,\ldots, N$. Then it suffices to assume in Theorem \ref{thm:general} that $c$, restricted to $E^N$, is Lipschitz. 
\end{remark}

We now show that $\pp_{\eps_n}(\bm^n)$ is equivalent to an \textsf{LP} problem. Hence, with a slight abuse of language, we always refer to $\pp_{\eps_n}(\bm^n)$ as the approximating \textsf{LP} problem of $\pp(\bm)$. 

\begin{corollary}\label{cor:lp}
Let $\bm^n=(\m^n_k)_{1\le k\le N}$ be chosen such that each $\m^n_k$ has finite support, i.e. 
\b*
\m^n_k(d\bx) ~=~ \sum_{i_k\in I_k} \alpha^k_{i_k} \delta_{\bx^k_{{i_k}}}(d\bx),
\e* 
where $I_k=\big\{1,\ldots, n(k)\big\}$ labels the support $\supp\big(\m^n_k\big)$. Denote by $p=\big(p_{i_1,\ldots, i_N}\big)_{i_1\in I_1,\ldots, i_N\in I_N}$ the elements of $ \R_+^D$ with $D:=\Pi_{k=1}^Nn(k)$, then $\pp_{\eps_n}(\bm^n)$ can be rewritten as an \textsf{LP} problem. 
\end{corollary}

\begin{proof}
By assumption, every element $\P\in\Mc_{\eps_n}(\bm^n)$ can be identified by some $p\in\R_+^D$. Therefore, $\pp_{\eps_n}(\bm^n)$ turns to be the optimization problem below
\begin{align}
& \max_{p\in\R_+^D}~ \sum_{i_1,\ldots, i_N}  p_{i_1,\ldots, i_N}  c\big(\bx^1_{i_1}, \ldots, \bx^N_{i_N}\big) \nonumber \\
\mbox{s.t.}\q & \sum_{i_1,\ldots, i_{k-1},i_{k+1},\ldots, i_N} p_{i_1,\ldots, i_N}  ~=~ \alpha^k_{i_k},~ \mbox{ for } i_k\in I_k \mbox{ and } k=1,\ldots, N, \label{pb:lp}\\
&\q \sum_{i_1,\ldots, i_k} \left |\sum_{i_{k+1},\ldots, i_N} p_{i_1,\ldots, i_N} \left(\bx^{k+1}_{i_{k+1}} - \bx^k_{i_k}\right) \right| ~ \le~  \eps_n,~ \mbox{ for } k=1,\ldots, N-1. \nonumber
\end{align}
\eqref{pb:lp} is not an  \textsf{LP} formulation, however, by  adding slack variables $\big(\delta^k_{i_1,\ldots, i_k, j}\big)_{i_1\in I_1,\ldots, i_k\in I_k, j\in J} \in\R_+^{D_k}$ with $J:=\big\{1,\ldots, d\big\}$ and $D_k:=d\Pi_{r=1}^kn(r)$, \eqref{pb:lp} is equivalent to the following \textsf{LP} problem:
\begin{align*}
 &\max_{p\in\R_+^D,~ \delta^1 \in \R_+^{D_1},~ \ldots,~ \delta^{N-1}\in \R_+^{D_{N-1}}}~ \sum_{i_1,\ldots, i_N}  p_{i_1,\ldots, i_N} c\big(\bx^1_{i_1}, \ldots, \bx^N_{i_N}\big) \\
 \mbox{s.t.} \q &  \sum_{i_1,\ldots, i_{k-1},i_{k+1},\ldots, i_N} p_{i_1,\ldots, i_N}  ~=~ \alpha^k_{i_k},~ \mbox{ for } i_k\in I_k \mbox{ and } k=1,\ldots, N, \\
& \q -\delta^k_{i_1,\ldots, i_k, j} ~\le~ \sum_{i_{k+1},\ldots, i_N} p_{i_1,\ldots, i_N}\left(x^{k+1}_{i_{k+1},j}-x^{k}_{i_k,j}\right) ~ \le~ \delta^k_{i_1,\ldots, i_k, j},~ \mbox{ for } i_k\in I_k, \\
\q & \q \q \q \q   j\in J \mbox{ and } k=1,\ldots, N, \\ 
\q & \q \q \sum_{i_1,\ldots, i_k, j} \delta^k_{i_1,\ldots, i_k,j} ~ \le~  \eps_n,~ \mbox{ for } k=1,\ldots, N-1,
\end{align*}
where we recall $\bx^k_{i_k}=\big(x^k_{i_k,1},\ldots, x^k_{i_k,d}\big)$. 
\end{proof}

Having obtained a general convergence result, we next turn to the ensuing problem on  
the convergence rate of $\pp_{\eps_n}(\bm^n)$. We provide an estimation of the convergence rate for the one-step model on real line. To the best of our knowledge, the error bound below is the first of its kind in the literature.

\begin{theorem}\label{thm:convergence_rate}
Let $N=2$ and $d=1$, or equivalently, $\bm=(\m,\n)$ and $c:\R^2\to\R$. In addition to the conditions of Theorem \ref{thm:general}, we assume that $\sup_{(x,y)\in\R^2}\big|\partial_{yy}^2c(x,y)\big|<\infty$ and  $\n$ has a finite second moment. Then there exists $C>0$ such that 
\b*
\big|\pp_{\eps_n}(\m^n,\n^n)-\pp(\m,\n)\big| ~ \le ~ C\inf_{R>0} \lambda_n(R),~ \mbox{ for all } n\ge 1,
\e*
where $\lambda_n: (0,\infty) \to\R$ is given by
\b*
\lambda_n(R) ~:=~ (R+1)\eps_n + \int_{(-\infty,-R)\cup (R,\infty)} \big(|y|-R\big)^2\n(dy).
\e*
In particular, the convergence rate is proportional to $\eps_n$ if $\supp(\n)$ is bounded.
\end{theorem}

We postpone the proofs of Theorems \ref{thm:general} and \ref{thm:convergence_rate} to Section \ref{sec:proof}, and end this section with a discussion about how Theorem \ref{thm:general} is applied to solve other constrained \textsf{OT} problems arising in finance. 

\begin{remark}\label{rem:mrh}
In general, the distributions $\m_1,\ldots, \m_N$ will not be fully specified by the market when $d\ge 2$. For $k=1,\ldots, N$, let $\bS_k:=\big(S^{(1)}_{k},\ldots, S^{(d)}_{k}\big)$, where $S^{(i)}_{k}$ stands for the price of the $i^{\textrm{th}}$ stock at time $k$. Then, in practice, only prices of call options $\big(S_{k}^{(i)}-K\big)^+$, or put options $\big(K-S_{k}^{(i)}\big)^+$, for a finite set of strikes $K$ are liquidly available in the market. 
Even assuming call options are quoted for all possible strikes $K$ only yields the distributions $\m_{k,i}$ of $S_{k}^{(i)}$. Therefore, this leads to a modified optimization problem. Denote $\vec{\m}_k:=(\m_{k,1},\ldots, \m_{k,d})$ and $\vec{\bm}:=(\vec{\m}_k)_{1\le k\le N}$, and let $\Mc_{\eps}(\vec{\bm})$ be the set of $\eps-$approximating martingale measures $\P$ satisfying $\P\circ \big(S_{k}^{(i)}\big)^{-1} = \m_{k,i}$,  for  $k=1,\ldots, N$  and  $i=1,\ldots, d$. Then we define the optimization problem by
\be\label{def:mmot}
\pp_{\eps}(\vec{\bm}) ~:=~ \sup_{\P\in \Mc_{\eps}(\vec{\bm})}~ \E_{\P}\big[ c(\bS_1,\ldots, \bS_N)\big].
\ee
The problem \eqref{def:mmot}, with $\eps=0$, was first introduced by Lim and called {\em multi-martingale optimal transport} in \cite{Lim2}. Although this paper focuses on the numerical computation of $\pp(\bm)$,  we emphasize that Theorem \ref{thm:general} admits an immediate extension to approximate $\pp_{0}(\vec{\bm})$. 
\end{remark}

\section{A numerical scheme for $\pp(\bm)$: probability discretization}\label{sec:num_pb}

Motivated by Theorem \ref{thm:general} and  Corollary \ref{cor:lp}, we next develop a numerical scheme to compute $\pp(\bm)$ based on a suitable discretization of the marginal distributions. The key is to select a suitable sequence $(\bm^n)_{n\ge 1}$ such that, for $k=1,\ldots, N$,

\begin{itemize}
\item[(a)] $\mu^n_k$ is supported on a finite set $\big\{\bx^k_{i_k}: i_k\in I_k\big\}$,

\vspace{1mm}

\item[(b)] the weights $\mu^n_k\big[\{\bx^k_{i_k}\}\big]$ can be either computed explicitly or approximated with a precision that is known {\em a priori}, 

\vspace{1mm}

\item[(c)] an upper bound for $\Wc\big(\m_k^n,\m_k\big)$ is easy to obtain.
\end{itemize}
Posed as above, the problem is intimately linked to the optimal quantization for probability measures whose goal is to best approximate a  given probability measure $\m\in\Pc$ by a discrete measure with a given number of supporting points. For the given $\mu$,  its  $n^d-$quantization $\m^n$ related to  $(\bx_i)_{1\le i\le n^d}\subset\R^d$ and  $(E_i)_{1\le i\le n^d}$ is defined by $\m^n(d\bx):=\sum_{i=1}^{n^d} \m[E_i] \delta_{\bx_i}(d\bx)$, where $(E_i)_{1\le i\le n^d}$ is a $\m-$partition, i.e. $\m\big[E_i\cap E_{j}\big]=0$ for all $i\neq j$ and $\m\big[\cup_{1\le i\le n^d} E_i\big]=1$. Accordingly, the $n^d-$optimal quantization of $\mu$ is the solution to 
\be\label{def:quantisation}
\inf~ \left\{\sum_{i=1}^{n^d} \int_{E_i} \big |\bx-\bx_i\big |\m(d\bx)\right\},
\ee 
where the inf is taken over all $(\bx_i)_{1\le i\le n^d}$ and $\m-$partitions $(E_i)_{1\le i\le n^d}$.  We state the convergence result from Graf and Luschgy  \cite{GL}, see also \cite{GLP, DFP}.

\begin{theorem}[Graf and Luschgy]\label{thm:quantisation}
For each $n\ge 1$, the inf in \eqref{def:quantisation} can be  achieved by an $n^d-$optimal quantizer $(\bx^*_i)_{1\le i\le n^d}$ and $(E^*_i)_{1\le i\le n^d}$. Let $\m^n_*$ be the corresponding optimizer, then $\lim_{n\to\infty}n\Wc(\m^n_*,\m)$ exists and is finite.
\end{theorem}

It follows with $A_{\m}:=\lim_{n\to\infty}n\Wc(\m^n_*,\m)$ that, there exists $n_{\m}\ge 1$ such that 
\b*
\Wc(\m^n_*,\m)~ \le~ (A_{\m}+1)/n,~ \mbox{ for all } n\ge n_{\m}.
\e*
Despite their theoretical appeal, in practice, the use of optimal quantizers is problematic as the key quantities above, such as $A_{\m}$ and $n_{\m}$ are in general unknown. Similarly, in general, the quantities $\m^n_*\big[\{\bx_i^*\}\big]$ are hard to compute exactly or approximate with a prescribed  accuracy. To overcome these difficulties, we adopt two different discretization methods, both of which can be implemented in practice. Our first method, which we call \emph{deterministic discretization}, applies when we are given the marginals $\m_1,\ldots, \m_N$ in the sense of being able to compute integrals against them. This is the case, e.g. when $\m_1,\ldots, \m_N$  have known density functions. The second method, called \emph{random discretization}, applies when we are able to sample from the marginals. Throughout Section \ref{sec:num_pb}, we need the following integrability condition, which corresponds to the market price of a power option being finite.

\begin{assumption}[$\theta^{\rm th}-$moment] \label{ass:mom}
There exist $\theta>1$ and $M_{\theta} < \infty$ such that 
\b*
\int_{\R^d} |\bx|^{\theta} \m_N(d\bx) ~ \le ~ M_{\theta}.
\e*
\end{assumption}

Note that by Jensen's inequality, the above conditions  implies $\int_{\R^d} |\bx|^{\theta} \m_k(d\bx) \le M_{\theta}$ for $k=1,\ldots, N$. Further, whenever we consider a generic measure $\mu$ below, we will also assume it satisfies Assumption \ref{ass:mom}.

\subsection{Deterministic discretization}\label{ssec:deter}

We devise a simple discretization procedure which has the same asymptotic efficiency as the optimal quantization when $\supp(\m)$ is bounded. We assume here that $\m$ is known in the sense that the probabilities $\m[E]$ are known for all $E\in\Bc(\R^d)$. We start with this idealized setting and then consider the case of known densities, which allows to compute $\m[E]$ with a certain accuracy.

\paragraph{Step 1: Truncation}
For $R>0$, let $\bB_R\subset\R^d$ denote the box defined by
\b*
\bB_R ~:=~ \Big\{\bx=(x_1,\ldots, x_d):\q  |x_i|\le R,~ \mbox{ for } i=1,\ldots, d\Big\}.
\e*
Then one has $\big\{\bx\in\R^d: |\bx|\le R\big\}\subset \bB_R\subset \big\{\bx\in\R^d: |\bx|\le dR\big\}$. Take  $R$ such that $\m[\bB_R]>0$, and truncate $\m$ into a probability measure $\m_R(d\bx) := \mathds{1}_{\bB_R}(\bx)\m(d\bx) + \m[\bB_R^c]\delta_{\0}(d\bx)$,  where $\bB_R^c:=\R^d\setminus \bB_R$. Clearly, $\m_R$ is supported on $\bB_R$. Consider a random variable $X$ drawn from $\m$  and observe that $\mathds{1}_{\bB_R}(X)X$ is distributed according to $\m_R$. We have, by the definition of Wasserstein distance,
\be\label{ineq:trunc}
\Wc(\m_R, \m) ~\le~  \E\big[\big|\mathds{1}_{\bB_R}(X)X-X\big|\big] ~ = ~ \int_{\bB_R^c} |\bx|\m(d\bx) ~ \le ~ M_{\theta}/R^{\theta-1},
\ee 
which yields in particular $\lim_{R\to\infty}\Wc(\m_R, \m)=0$. 

\paragraph{Step 2: Discretization}
Denote by $\Om_n\subset \R^d$ the countable  subspace consisting of elements $ \bq/n$ for all $\bq=(q_1,\ldots, q_d)\in\Z^d$. For each $\bq\in\Z^d$, we denote by $V(\bq/n)\subset\R^d$ the subset of $\bx=(x_1,\ldots, x_d)$ such that $\lfloor n\bx \rfloor=\bq$, i.e. $\lfloor nx_i \rfloor  =q_i$ for $i=1,\ldots, d$, where for $a\in\R$, $\lfloor a \rfloor\in\Z$ is the largest integer less or equal to $a$. We construct a probability measure $\m^{(n)}$ whose support is  included in  $\Om_n$ by $\m^{(n)}\big[\{\bq/n\}\big]:= \m\big[V(\bq/n)\big]$.  
Then $\m^{(n)}\in\Pc$ satisfies for all $f\in\Lambda$, 
\be\label{eq:discret}
\int f d\m^{(n)} ~=~ \sum_{\bq\in\Z^d} f(\bq/n) \m^{(n)}\big[\{\bq/n\}\big] ~
=~ \int f^{(n)} d\m,
\ee
where $f^{(n)}:\R^d\to \R$ is defined by $f^{(n)}(\bx):= f\big(\lfloor n\bx \rfloor/n \big)$. This implies in  view of \eqref{def:wass_dual} that
\b*
\Wc\big(\m^{(n)},\m\big) ~=~ \sup_{f\in\Lam_1}~ \left | \int  f d\m^{(n)} - \int f d\m \right | ~\le~ \sup_{f\in\Lam_1}~  \int \big|f^{(n)}-f \big|  d\m ~\le ~ d/n,
\e*
where the second inequality is by \eqref{eq:discret}. Notice that, if $\supp(\m)$ is bounded, then so is $\supp\big(\m^{(n)}\big)$, and the distance $\Wc\big(\m^{(n)}, \m\big)$ is of order $1/n$, which is the same as for $\Wc(\m_*^{n}, \m)$. 

\paragraph{Step 3: Choice of the parameters}
Replacing $\m$ by $\m_R$ in {\em Step 2},  one has 
\b*
\Wc\big(\m_R^{(n)},\m\big) ~\le~ \Wc\big(\m_R^{(n)},\m_R\big) + \Wc(\m_R,\m) ~\le ~ d/n  + M_{\theta}/R^{\theta-1}.
\e*
It follows from Young's inequality that $
|a|^{\gamma} + |b|^{\theta} \ge \gamma^{1/\gamma}\theta^{1/\theta}|ab|$ for all $a$, $b\in \R$,  where $\gamma>1$ is the conjugate number of $\theta$, i.e. $1/\theta+1/\gamma=1$. Setting respectively $a=(d/n)^{1/\gamma}$ and $b=\big(M_{\theta}/R^{\theta-1}\big)^{1/\theta}$, it holds that $
d/n  + M_{\theta}/R^{\theta-1} \ge (\gamma d)^{1/\gamma}(\theta M_{\theta})^{1/\theta}/(Rn)^{1/\gamma}$, and the equality can be achieved for $R^{\theta-1}=\theta M_{\theta}n/\gamma d$. Since the cardinal of $\supp\big(\m_R^{(n)}\big)$ is proportional to $(Rn)^d$ which determines the number of variables in the corresponding \textsf{LP} problem, setting $R=R_n:=\big(\theta M_{\theta}n/\gamma d\big)^{1/(\theta-1)}$ leads to an optimal upper bound for a fixed computational complexity, i.e. $\Wc\big(\m^{(n)}_{R_n}, \m\big) \le  \gamma d/n$. Replacing $\m$ respectively by $\m_k$ for $k=1, \ldots, N$, we obtain $\bm^n=(\m_k^n)_{1\le k\le N}$  following the above steps with $\m^n_k:=\m^{(n)}_{k,R_n}$. Then Theorem \ref{thm:general} yields $
\lim_{n\to\infty}\pp_{N\gamma d/ n}(\bm^n) = \pp(\bm)$. 

\begin{remark}\label{rk:discretization_convexorder}
In general $\bm^n$ may no longer belong to $\Pc^{\pq}$, even if $\bm\in\Pc^{\pq}$. When $d=1$,  an explicit discretization preserving the increasing convex order is given in Section \ref{ssec:convergence}. In a recent parallel work  Alfonsi et al.\ \cite{ACJ} have investigated methods of constructing $\bm^n$ such that $\bm^n\in \Pc^\pq$. 
\end{remark}

The above analysis allows us to construct approximating measures $\m^n$ assuming the values $\m\big[V(\bq/n)\big]$ are known for all $\bq/n\in\Om_n$. This may be possible, e.g. when $\m$ is  atomic, but in general we need to argue how to approximate well such values. We do this for measures which admit a density function, i.e. $\m(d\bx)=\rho(\bx)d\bx$. In this case, a simple point estimate $\rho(\bx_{\bq})/n^d$, for some $\bx_{\bq}\in V(\bq/n)$, provides a natural candidate to approximate $\m^{(n)} \big[\{\bq/n\}\big]$. However, to use Theorem \ref{thm:general}, we need to bound the Wasserstein distance between the resulting measure and $\m^{(n)}$ in an explicit and non-asymptotic manner.

As before, we truncate $\m$ to $\bB_{R}$ and set $R$ to be an integer $m$ for simplicity. Let $\m_m^{(n)}$ and $\tmn$ be supported on $\Om_n\cap \bB_m$ and defined as follows: If $\0\neq \bq/n \in \bB_m$, then
\b*
\mn_m\big[\{\bq/n\}\big] ~:=~
\int_{V(\bq/n)} \rho(\bx)d\bx &\mbox{and}& \tmn\big[\{\bq/n\}\big] ~:=~ \rho(\bx_{\bq})/n^d,
\e*
where $\bx_{\bq}\in V(\bq/n)$ are chosen arbitrarily, and
\b*
\mn_m\big[\{\0\}\big] ~:=~ 1 - \sum_{\bq/n\neq \0} \mn_m\big[\{\bq/n\}\big]
&\mbox{and}&
\tmn\big[\{\0\}\big] ~:=~ 1 - \sum_{\bq/n\neq \0}  
 \tmn\big[\{\bq/n\}\big],
\e*
where the sums above are indeed finite as 
$\mn_m\big[\{\bq/n\}\big]=\tmn\big[\{\bq/n\}\big] =0$ for $\bq/n \notin \bB_m$. As above, $\Wc\big(\m_m^{(n)},\m\big) \le d/n  + M_{\theta}/m^{\theta-1}$ and the following gives an upper bound for $\Wc\big(\tmn, \m_m^{(n)}\big)$.  

\begin{proposition}\label{prop:estimation}
Let Assumption \ref{ass:mom} hold. Suppose that  $\rho$ is continuous, or namely, for each $R>0$, there exists $\kappa_R: [0,\infty)\to \R $ that is non-decreasing such that $\kappa_R(0)=0$ and 
\b*
\big|\rho(\bx) -\rho(\by) \big| ~\le~  \kappa_R(|\bx-\by|),\q  \mbox{for all }  \bx,~ \by\in\bB_R.
\e* 
\rmi If $\m$ has bounded support, i.e. $\supp(\m)\sb \bB_R$ for some $R>0$, then
\be\label{ineq:num0}
~~~~~ \Wc\big(\tmn, \m\big)  ~\le~  \eps_{n}  ~:=~ d/n  +  2^dd(R+1)^{d+1}\kappa_{R+1}(d/n),~ \mbox{ for all } m\ge \lfloor R \rfloor+1.
\ee
\rmii  If $\rho$ is uniformly continuous, i.e. there exists a uniform $\kappa=\kappa_R$ for all $R>0$, then
\be\label{ineq:num1}
~~~ \Wc\big(\tmn, \m\big)  ~\le~  \eps_{m,n}  ~:=~ d/n  + M_{\theta}/m^{\theta-1} + 2^ddm^{d+1}\kappa(d/n).
\ee
\no \rmiii If $\{\bx_{\bq}\}_{\0\neq \bq/n\in \bB_m}$ satisfies $\rho(\bx_{\bq})\le \rho(\bx)$ for all $\bx\in V(\bq/n)$, then
\be\label{ineq:num2}
~~~~~~ \Wc\big(\tmn, \m\big)  ~ \le~ \tau_{m,n} ~:=~ d/n  + M_{\theta}/m^{\theta-1} + \inf_{1\le j\le m} \Big\{2^d d j^{d+1}\kappa_{j}(d/n) + 4M_{\theta}/j^{\theta-1}\Big\}.
\ee
\end{proposition}

\begin{proof}
Under the condition {\rm (i)} one has $\m=\m_m$ as $m\ge \lfloor R \rfloor+1$, and thus $\Wc\big(\m_m^{(n)},\m\big)=\Wc\big(\m_m^{(n)},\m_m\big)\le d/n$. Note also $\supp\big(\m_m^{(n)}\big)$, $\supp\big(\tmn\big) \sb \bB_{\lfloor R \rfloor+1}$ by definition. For any $f\in\Lambda_1$, it holds for $m\ge \lfloor R \rfloor+1$
\b*
&& \left | \int  f d\tmn - \int f d\m^{(n)}_m  \right | 
~~=~~ \left | \sum_{\0\neq \bq/n \in \bB_{\lfloor R \rfloor+1}} f(\bq/n)\int_{V(\bq/n)} \Big(\rho(\bx)-\rho(\bx_{\bq})\Big) d\bx\right | \\
&& \le~~ \left | \sum_{\0\neq \bq/n \in \bB_{\lfloor R \rfloor+1}} |\bq/n| \int_{V(\bq/n)} \kappa_{R+1}(d/n) d\bx\right | ~~\le~~ 2^dd(R+1)^{d+1}\kappa_{R+1}(d/n),
\e*
which yields $\Wc\big(\tmn, \m_m^{(n)}\big)  \le 2^dd(R+1)^{d+1}\kappa_{R+1}(d/n)$ and further \eqref{ineq:num0}. As for {\rm (ii)}, we deduce $\Wc\big(\tmn, \m_m^{(n)}\big)  \le 2^ddm^{d+1}\kappa(d/n)$ using the same arguments for $\bB_m$, and obtain thus \eqref{ineq:num1}. Alternatively, assume that the third condition holds. For each integer $1\le j\le m$, one has
\b*
&&\left | \int f d\tmn - \int f d\m^{(n)}_m  \right | \\
&\le& \left | \sum_{\0\neq \bq/n \in \bB_j} |\bq/n| \int_{V(\bq/n)} \Big(\rho(\bx)-\rho(\bx_{\bq})\Big) d\bx\right | +  \left | \sum_{\bq/n \in \bB_m\setminus B_j} |\bq/n| \int_{V(\bq/n)} \Big(\rho(\bx)-\rho(\bx_{\bq})\Big) d\bx\right |  \\
&\le& 2^d d j^{d+1}\kappa_{j}(d/n) + 2\int_{\bB_{j}^c} \big(d+|\bx|\big) \rho(\bx)d\bx ~~\le~~ 2^dd j^{d+1}\kappa_{j}(d/n) + 4M_{\theta}/j^{\theta-1}.
\e*
Thus $\Wc\big(\tmn, \m_m^{(n)}\big) \le \inf_{1\le j\le m} \big\{2^dd j^{d+1}\kappa_{j}(d/n) + 4M_{\theta}/j^{\theta-1}\big\}$ follows and \eqref{ineq:num2} is derived.
\end{proof}

By a straightforward computation, one has $\lim_{m\to\infty}\lim_{n\to\infty} \eps_{m,n}= \lim_{m\to\infty}\lim_{n\to\infty} \tau_{m,n}=0$. In consequence, there exist suitable sequences $(m_n)_{n\ge 1}$ and $(n_m)_{m\ge 1}$, such that 
\b*
\lim_{n\to\infty}\eps_{m_n,n}=\lim_{n\to\infty}\tau_{m_n,n}=0 &\mbox{and}& \lim_{m\to\infty}\eps_{m,n_m}= \lim_{m\to\infty}\tau_{m,n_m}=0.
\e*
Note that the previous choice $m_n:=\lfloor R_n \rfloor$ may not yield $\lim_{n\to\infty} \eps_{\lfloor R_n \rfloor,n}= \lim_{n\to\infty} \tau_{\lfloor R_n \rfloor,n}=0$ and these sequences have to be  computed from $\rho$. However, if $\rho$ is $L-$Lipschitz, then one has $\eps_{m,n}=d/n  + M_{\theta}/m^{\theta-1} + 2^dd^2m^{d+1}L/n$ and we deduce that it suffices to take $(m_n)_{n\ge 1}$ or $(n_m)_{m\ge 1}$ such that $\lim_{n\to\infty}m_n^{d+1}/n=0$ or $\lim_{m\to\infty}m^{d+1}/n_m=0$. 

Putting everything together, taking respectively $\m_k$ in the place of $\m$ for $k=1, \ldots, N$, the above procedures yield a vector of measures $\tilde{\bm}^{(n)}_m=\big(\tilde{\m}_{k,m}^{(n)}\big)_{1\le k\le N}$. Under the conditions in Proposition \ref{prop:estimation}, we have $\lim_{n\to\infty}\pp_{N\eps_{m_n,n}}\big(\tilde{\bm}^{(n)}_{m_n}\big) = \lim_{n\to\infty}\pp_{N\tau_{m_n,n}}\big(\tilde{\bm}^{(n)}_{m_n}\big)= \pp(\bm)$ or $\lim_{m\to\infty}\pp_{N\eps_{m,n_m}}\big(\tilde{\bm}^{(n_m)}_{m}\big) = \lim_{m\to\infty}\pp_{N\tau_{m,n_m}}\big(\tilde{\bm}^{(n_m)}_{m}\big)= \pp(\bm)$.

\subsection{Random discretization}
\label{ssec:rand}
We consider now a different discretization procedure, which applies to the case where one has a black box to generate independent random variables according to $\m$. Provided a sequence of i.i.d. $\mu-$distributed random variables $(X_n)_{n\ge 1}$, define the empirical measure $\hm$ by 
\b*
\hm(d\bx) ~:=~\sum_{i=1}^n \frac{1}{n}\delta_{X_i}(d\bx).\e*
By definition $\hm$ is a random measure, and following  Glivenko-Cantelli's theorem, see e.g. Fournier and Guillin \cite{FG}, $\lim_{n\to\infty}\Wc(\hm,\m)=0$ almost surely and $\lim_{n\to\infty}\E\big[\Wc(\hm,\m)\big]=0$. Construct random measures $\hm_k$ by replacing $\m$ by $\m_k$ for $k=1,\ldots, N$ and set $\hbm:=(\hm_k)_{1\le k\le N}$. Compared to Theorem \ref{thm:general}, we now obtain a stochastic convergence result. 

\begin{proposition}\label{prop:rand}
Let the conditions of Theorem \ref{thm:general} hold. Given a sequence $(\eps_m)_{m\ge 1}\subset (0,\infty)$ converging to zero, one has $\lim_{m\to\infty} \lim_{n\to\infty} \pp_{\eps_m}(\hbm) = \pp(\bm)$  almost surely. Further, for any subsequence $(\hnm)_{m\ge 1}$ such that $\sum_{m\ge 1} \E\big[\Wc^{\oplus}(\hat{\bm}^{\hnm},\bm)\big]/\eps_m <\infty$, it holds almost surely $\lim_{m\to\infty} \pp_{\eps_m}(\hat{\bm}^{\hnm})=\pp(\bm)$.\end{proposition}

\begin{proof}
For each fixed $\eps_m>0$, one has the inequality below by Corollary \ref{cor1}  
\be\label{ineq:rand}
~~~~~~ \big | \pp_{\eps_m}(\hbm) - \pp(\bm)\big| ~\le ~ {\li}(c)\eps_m + \pp_{2\eps_m}(\bm) - \pp(\bm)  \mbox{ whenever } \Wc^{\oplus}(\hbm,\bm)\le\eps_m,
\ee
where we recall that ${\li}(c)$ is the Lipschitz constant of $c$. Applying Glivenko-Cantelli's theorem, one obtains $
\lim_{n\to\infty} \big |\pp_{\eps_m}(\hbm) - \pp(\bm) \big| \le  {\li}(c)\eps_m + \pp_{2\eps_m}(\bm) - \pp(\bm)=:\delta_m$ almost surely. Using Proposition \ref{prop:rusc} we have $\lim_{m\to\infty}\pp_{2\eps_m}(\bm) =\pp(\bm)$ and the first asserted convergence result follows. For any $\delta>0$, there exists $m_{\delta}$ such that $\eps_m\le \delta$ for all $m\ge m_{\delta}$. Taking $\hnm$ as in the statement, we have  
\b*
\sum_{m\ge m_{\delta}} \P\Big[\big|\pp_{\eps_m}(\hat{\bm}^{\hnm}) - \pp(\bm)\big|>\delta\Big] 
\le  \sum_{m\ge m_{\delta}} \P\Big[\Wc^{\oplus}\big(\hat{\bm}^{\hnm}, \bm\big)>\eps_m\Big]  \le 
\sum_{m\ge m_{\delta}} \E\big[\Wc^{\oplus}(\hat{\bm}^{\hnm},\bm)\big]/\eps_m,
\e*
which, by Borel-Cantelli's lemma, implies that $\lim_{m\to\infty} \pp_{\eps_m}(\hat{\bm}^{\hnm})=\pp(\bm)$ almost surely.
\end{proof}

Roughly speaking, given $\eps_m>0$, it suffices to focus on the \textsf{LP} problems $\pp_{\eps_m}(\hbm)$  such that $\Wc^{\oplus}(\hbm,\bm)\le \eps_m$ occurs with high probability. 
To this end, and in order to choose suitable sequences $(\eps_m)_{m\ge 1}$ and $(\hnm)_{m\ge 1}$, we need to quantify $\E\big[\Wc^{\oplus}(\hbm,\bm)\big]$. Fortunately, Theorem 1 in \cite{FG} provides such an estimation under Assumption \ref{ass:mom}, albeit with some cases omitted, e.g. $d=1, 2$ and $\theta=2$. For the sake of completeness, we state this result as Lemma \ref{prop:sto_disc} by taking all the cases into account.

\begin{lemma}[\cite{FG}]\label{prop:sto_disc}
Let Assumption \ref{ass:mom} hold. There exists  $C(\theta,d)>0$ such that $\E\big[\Wc^{\oplus}\big(\hbm,\bm\big)\big]  \le \chi(n)$ for all $n\ge 1$, where
\b*
\chi(n) ~:=~ NC(\theta, d) \begin{cases}
n^{1/\theta-1} &~   \mbox{ if }  d=1 \mbox{ and } 1<\theta< 2, \\
(1+\log n)  n^{-1/2} &~   \mbox{ if }  d=1 \mbox{ and } \theta = 2, \\
n^{-1/2} &~   \mbox{ if }  d=1 \mbox{ and } \theta > 2, \\
n^{1/\theta-1} &~   \mbox{ if }  d=2 \mbox{ and } 1<\theta< 2, \\ 
\big(1+(\log n)^2 \big) n^{-1/2}  &~   \mbox{ if }  d=2 \mbox{ and } \theta = 2, \\ 
(1+\log n) n^{-1/2}  &~   \mbox{ if }  d=2 \mbox{ and } \theta > 2, \\ 
n^{1/\theta-1} &~  \mbox{ if }  d\ge 3 \mbox{ and } 1<\theta < d/(d-1), \\
(1+\log n) n^{-1/d}  &~  \mbox{ if }  d\ge 3 \mbox{ and } \theta = d/(d-1), \\
n^{-1/d} &~  \mbox{ if }  d\ge 3 \mbox{ and } \theta > d/(d-1).
\end{cases}
\e*
In consequence, one has $\P\big[\Wc^{\oplus}\big(\hbm,\bm\big)> \eps_m\big]  \le  \chi(n)/\eps_m$. 
\end{lemma}

We remark that, the constants $C(\theta, d)$ are not explicitly specified in \cite{FG}. To implement our scheme we need to determine $\hnm$ and for this we have to compute explicitly $C(\theta, d)$.  This is possible, following largely the arguments in \cite{FG}, but tedious and is postponed to Appendix  \ref{sec:append}.

\subsection{Numerical examples}\label{ssec:example}

We discuss now some concrete \textsf{MOT} problems to illustrate how Theorem \ref{thm:general}, together with our discretization schemes, can be applied. Most of our examples admit either a closed-form optimizer or an analytical characterization thereof, which allow us to verify the numerical results. We recall that for $d=1$, we write simply $\bx=x$ and $\bS_k=S_k$.

\begin{example}\label{ex1}
Beiglb\"ock and Juillet studied in \cite{BJ} a specific \textsf{MOT} problem in the case of $N=2$, $d=1$ and $c(x,y)=h(x-y)$, where $h:\R\to\R$ has a strictly convex  derivative. Let $(\m,\n)\in\Pc^{\pq}$.  Theorem 1.7 of  \cite{BJ} shows that, if $\m$ has a density $\rho$, then there exist two measurable functions  $\xi_{\pm}:\R\to\R$ such that the unique optimizer $\P^*\in\Mc(\m,\n)$ for $\pp(\m,\n)$ is supported on $\xi_{\pm}$, i.e. 
\b*
\P^*(dx,dy)~ = ~ \m(dx)\ox \left\{\frac{x-\xi_-(x)}{\xi_+(x)-\xi_-(x)} \delta_{\xi_{+}(x)} (dy) + \frac{\xi_+(x)-x}{\xi_+(x)-\xi_-(x)} \delta_{\xi_{-}(x)}(dy) \right\},
\e*
where  $\xi_-(x)\le x\le \xi_+(x)$, $\xi_+(x)<\xi_+(x')$ and $\xi_-(x')\notin \big(\xi_-(x),\xi_+(x)\big)$ for all $x$, $x'\in\R$ with $x<x'$. We want to illustrate numerically the above result. Let $\rho$ be a truncated Gamma function defined by
\b*
\rho(x) ~:=~ {\mathds 1}_{[0,1]}(x) x^{3/2}e^{-x} /C,~ \mbox{ where } C:= \int_0^1 x^{3/2}e^{-x}dx > 1/5.
\e*
Next construct $\n(dx)=\sigma(y)dy$ by 
\b*
\sigma(y)~:=~\rho(y/2)/6 + 4\rho(2y)/3 ~=~  {\mathds 1}_{[0,2]}(y) (y/2)^{3/2}e^{-y/2} /C+ {\mathds 1}_{[0,1/2]}(y) (2y)^{3/2}e^{-2y} /C.
\e*
By the construction, one has $(\m,\n)\in\Pc^{\pq}$, $\supp(\m)=[0,1]$ and $\supp(\n)=[0,2]$. 
Further, one can verify that $\rho$ and $\sigma$ are $L-$Lipschitz on $\supp(\m)$ and $\supp(\n)$ with $L=7$. 
Applying the discretization of Section \ref{ssec:deter} to $\m$ and $\n$, we obtain $\tn$ and $\tm$, supported on $\big\{i/n: 0\le i< n \big\}$ and $\big\{j/n: 0\le j< 2n \big\}$, and defined by 
\b* 
\tn\big[\big\{i/n\big\}\big] &:=& \begin{cases}
1- \sum_{k=1}^{n-1} \rho(x_k)/n &  \mbox{ if }  i=0, \\
\rho(x_i)/n &  \mbox{ if }  1\le i < n,
\end{cases} \\
 \tm\big[\big\{j/n\big\}\big] &:=& \begin{cases}
1- \sum_{k=1}^{2n-1} \sigma(y_k)/n &  \mbox{ if }  j=0, \\
\sigma(y_j)/n & \mbox{ if }  1 \le j< 2n,
\end{cases}
 \e* 
 where $x_i\in \big[i/n,(i+1)/n\big)$ and $y_j\in \big[j/n,(j+1)/n\big)$ for $i=1,\ldots, n-1$ and $j=1,\ldots, 2n-1$. It follows from Proposition \ref{prop:estimation} that $
 \Wc^{\oplus}\big((\tn,\tm),(\m,\n)\big) \le (3L+2)/n=:\eps_n$. 
Then the corresponding \textsf{LP} problem is as follows:
\begin{align*}
 \max_{(p_{i,j}) \in\R_+^{2n^2}}~ \sum_{i=0}^{n-1}  \sum_{j=0}^{2n-1} p_{i,j}h\big((i-j)/n\big)
\q \mbox{s.t. } \q & \sum_{j=0}^{2n-1}  p_{i,j} ~=~ \alpha^n_i, ~ \mbox{ for } i=0,\ldots, n-1,\\
& \sum_{i=0}^{n-1} p_{i,j} ~=~\beta^n_j,~ \mbox{ for } j=0,\ldots, 2n-1,\\
  & \sum_{i=0}^{n-1} \left|\sum_{j=0}^{2n-1}  p_{i,j}j/n - \alpha_i^n i/n\right|~\le~  \eps_n,
  \end{align*}
  where $\alpha^n_i:=\tn\big[\big\{i/n\big\}\big]$ and $\beta^n_j:=\tm\big[\big\{j/n\big\}\big]$.
Taking $h(x):=e^x$, we solve the \textsf{LP}  problem using \textsf{Gurobi} solver and present the results in  Figure \ref{fg:BJ}. 
\begin{figure}[h]
  \begin{minipage}[b]{0.45\textwidth}
\includegraphics[width=\textwidth,height=0.68\textwidth]{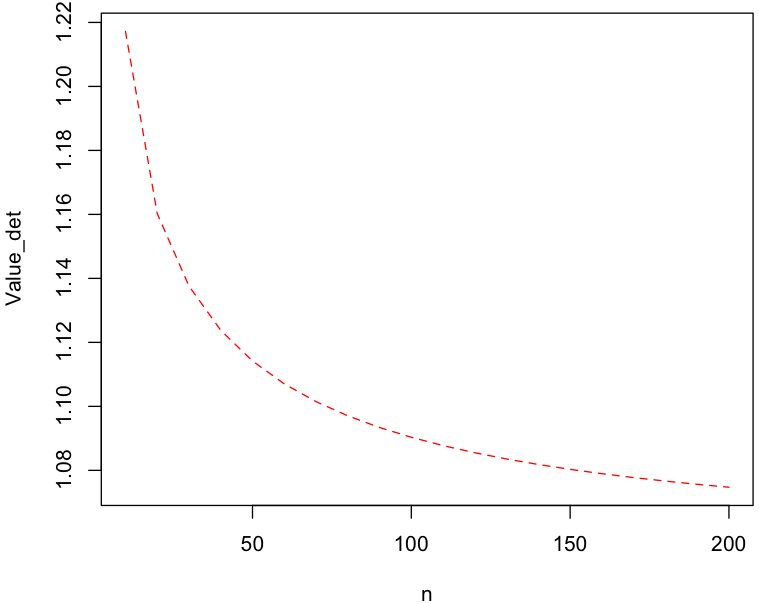}
  \end{minipage}
    \hfill
  \begin{minipage}[b]{0.45\textwidth}  \includegraphics[width=\textwidth,height=0.67\textwidth]{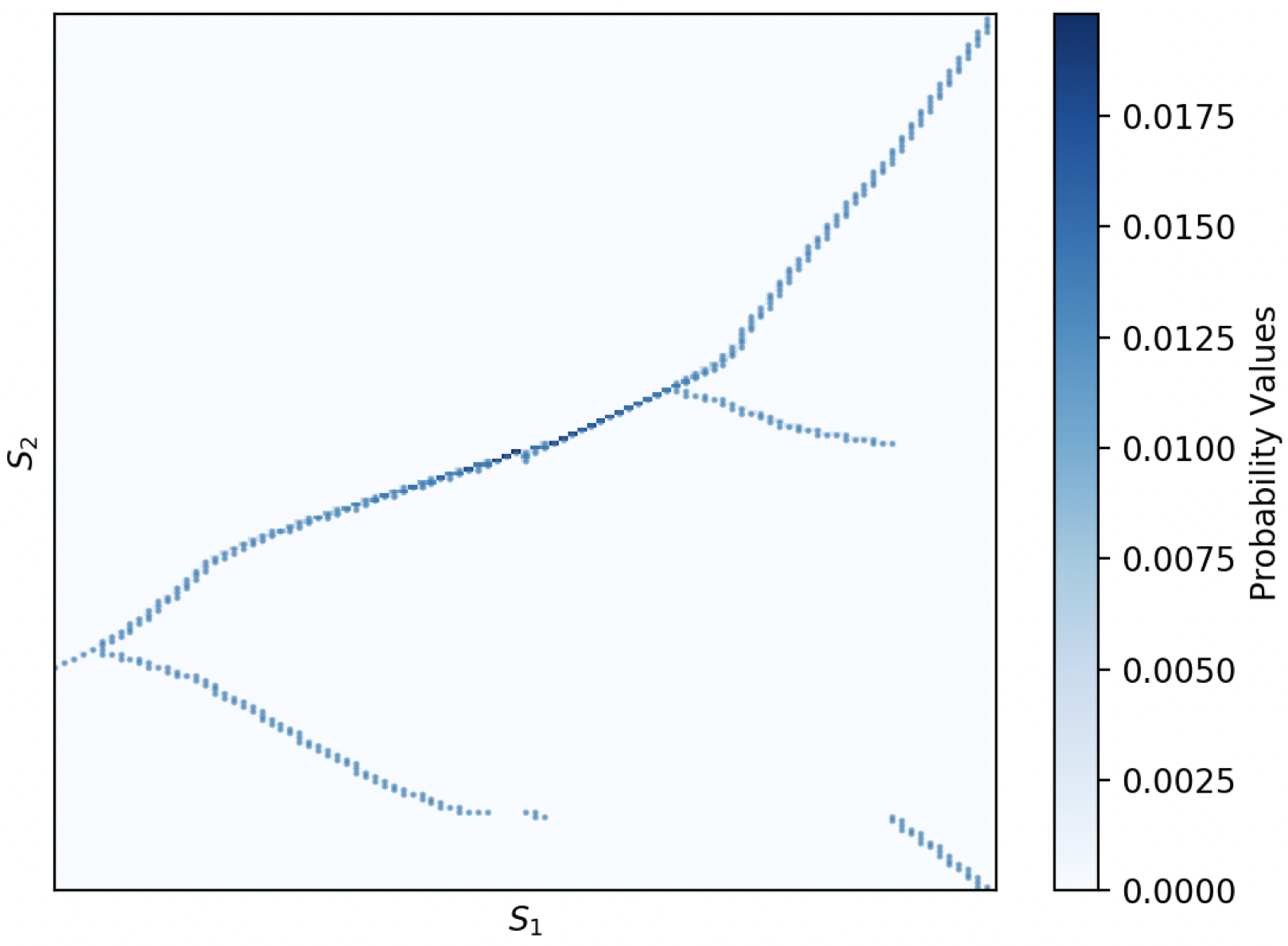}
  \end{minipage}
  \caption{Computations for Example \ref{ex1}. The first pane shows the values $\pp_{\eps_n}(\tn,\tm)$ for $10\le n\le 200$. The second pane draws the heat map of the optimizer for $n=100$. }\label{fg:BJ}
\end{figure}
The left pane exhibits the values $\pp_{\eps_n}(\tn,\tm)$ for $10\le n\le 200$, which shows numerically  the convergence of $\pp_{\eps_n}(\tn,\tm)$. The right pane displays the heat map of the optimizer $(p^*_{i,j})$ for $n=100$. We see that the strictly positive weights $p^*_{i,j}$ are concentrating around two curves that satisfy the conditions of $\xi_{\pm}$ stated above. 

For comparison, we adopt now the random discretization developed in Section \ref{ssec:rand}. We sample, using an accept-reject algorithm, two sequences of i.i.d.\ random variables $(X_i)_{1\le i\le n}$ and $(Y_j)_{1\le j\le n}$ from $\m$ and $\n$ respectively. Let $\Xc$ and $\Yc$ be two sets containing respectively all values taken by $X_i$ and $Y_j$, and we relabel $\Xc$ and $\Yc$ by $\Xc := \big\{X^{1}, X^2, \cdots, X^{\#\Xc}\big\}$ and $\Yc := \big\{Y^{1}, Y^2, \cdots, Y^{\#\Yc}\big\}$, where $\#\Xc$, $\#\Yc \le n$ denote the cardinal of $\Xc$ and $\Yc$. Define further $\hm(dx):=\sum_{i=1}^{\#\Xc}\hat{\alpha}^n_i\delta_{X^i}(dx)$ and $\hn(dy):=\sum_{j=1}^{\#\Yc}\hat{\beta}^n_i\delta_{Y^j}(dy)$, where
\b*
\hat{\alpha}^n_i ~:=~ \frac{\#\Xc_i}{n} \mbox{ for } i=1,\ldots, \#\Xc &\mbox{and}& \hat{\beta}^n_j ~:=~ \frac{\#\Yc_j}{n} \mbox{ for } j=1,\ldots, \#\Yc,
\e*
with $\Xc_i:=\big\{X_k\in\Xc: X_k=X^i\big\}$ and $\Yc_j:=\big\{X_k\in\Yc: X_k=Y^j\big\}$. The \textsf{LP} problem $\pp_{\eps_m}(\hm,\hn)$ 
is given by
\begin{align*}
 \max_{(p_{i,j})\in\R_+^{\#\Xc\#\Yc}}~ \sum_{i=1}^{\#\Xc}\sum_{j=1}^{\#\Yc} p_{i,j}h(X^i - Y^j) 
\q \mbox{s.t. } \q &  \sum_{j=1}^{\#\Yc} p_{i,j} ~=~ \hat{\alpha}^n_i,~ \mbox{ for } i=1,\ldots, \#\Xc,\\
&  \sum_{i=1}^{\#\Xc} p_{i,j} ~=~ \hat{\beta}^n_j,~  \mbox{ for }  j=1,\ldots, \#\Yc,\\
  &\sum_{i=1}^{\#\Xc} \left| \sum_{j=1}^{\#\Yc} p_{i,j}Y^j - \hat{\alpha}^n_iX^i \right| ~\le~ \eps_m.
 \end{align*}
 Notice that $\n$ admits a finite $\theta^{\rm th}-$moment for all $\theta>1$. With $\theta=3$, one has $\chi(n)= 2C(3,1)n^{-1/2}$, where $C(3,1)$ is defined in Proposition \ref{prop:sto_disc}. We set $\hnm:=\lfloor m^{r} \rfloor$ so that $\sum_{m\ge 1} \chi(\hnm)/\eps_m <\infty$ whenever $r>4$, and hence $\lim_{m\to\infty}\pp_{\eps_m}(\hat{\m}^{\hnm}, \hat{\n}^{\hnm})=\pp(\m, \n)$ holds almost surely. Taking $r=4.1$, we compute $\pp_{\eps_m}(\hat{\m}^{\hnm}, \hat{\n}^{\hnm})$ and present the results in Figure \ref{fg:BJ_sto}. 
 \begin{figure}[h]
  \begin{minipage}[b]{0.47\textwidth}
\includegraphics[width=\textwidth,height=0.68\textwidth]{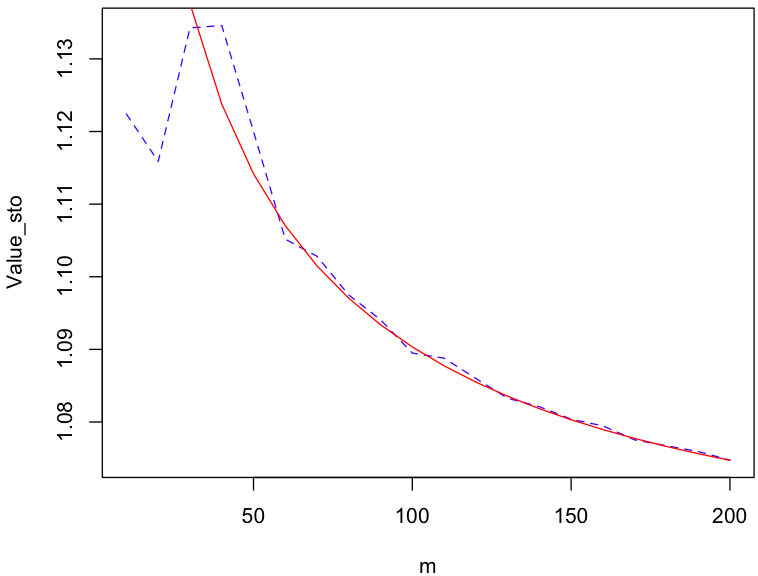}
  \end{minipage}
    \hfill
  \begin{minipage}[b]{0.47\textwidth}  \includegraphics[width=\textwidth,height=0.67\textwidth]{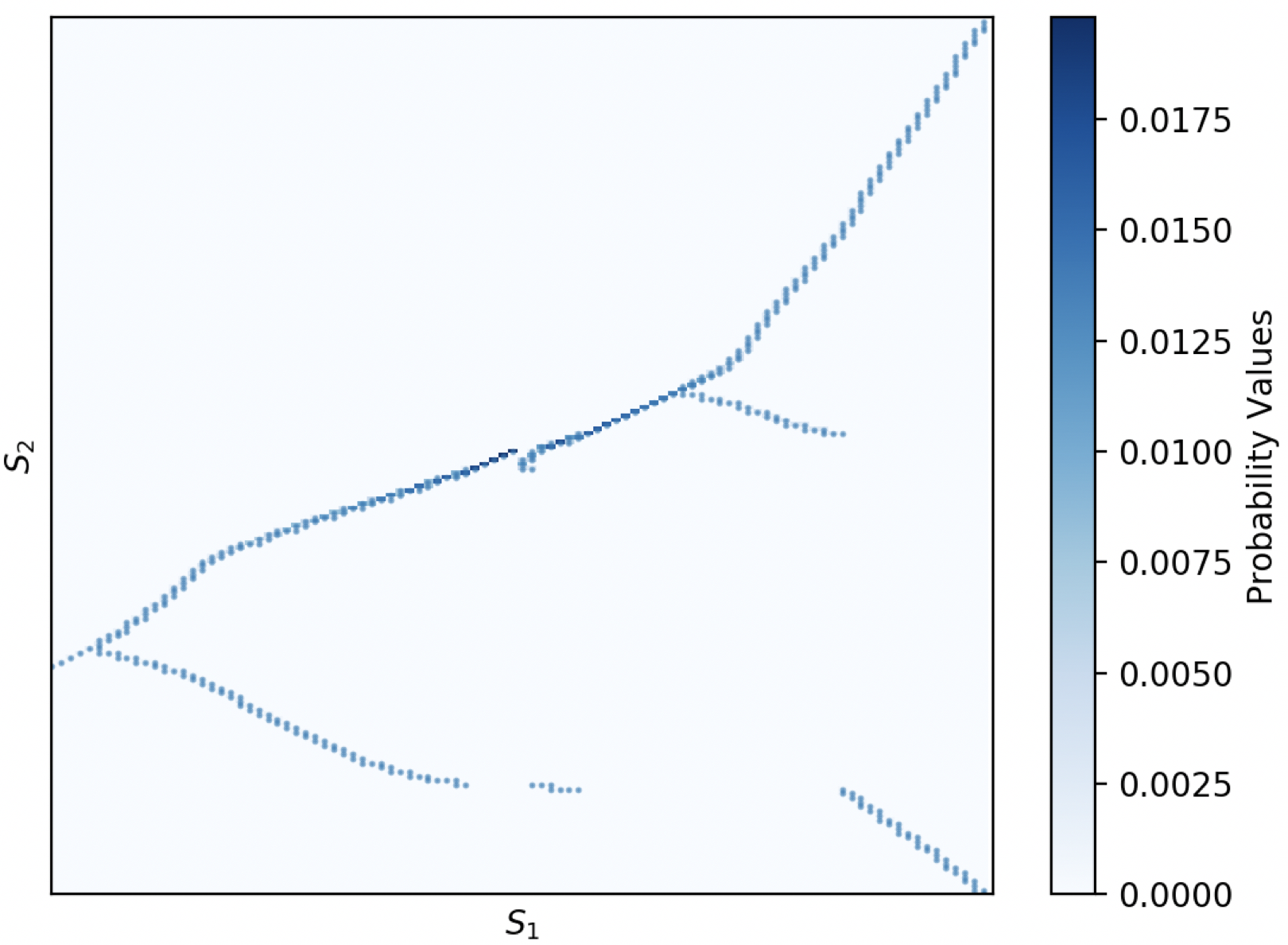}
  \end{minipage}
  \caption{Computations for Example \ref{ex1}. The first pane shows the values $\pp_{\eps_m}(\hat{\m}^{\hnm}, \hat{\n}^{\hnm})$ (dashed line) for $10\le m\le 200$. The second pane draws the heat map of the optimizer for $m=100$.}\label{fg:BJ_sto}
\end{figure}
The blue line in the left pane shows the convergence of $\pp_{\eps_m}(\hat{\m}^{\hnm}, \hat{\n}^{\hnm})$ in $m$ while the red line reproduces the convergence of $\pp_{\eps_m}(\tilde{\m}^{(m)},\tilde{\n}^{(m)})$ from the first pane of Figure \ref{fg:BJ}. While the random discretization displays some instability for small $m$, for $m\geq 50$ we we find that two lines are very close. We note that for the same $\eps_m$, the number of variables for $\pp_{\eps_m}(\tilde{\m}^{(m)},\tilde{\n}^{(m)})$ is proportional to $m^2$ while it is of order  $m^{2r}\gg m^2$ for $\pp_{\eps_m}(\hat{\m}^{\hnm}, \hat{\n}^{\hnm})$. In the right pane, the heat map of the optimizer for $m=100$ is drawn and matches closely that of Figure \ref{fg:BJ}. 
\end{example}

\begin{example}\label{ex2}
Motivated by the model-independent pricing, we consider a stock with three trading dates, i.e. $N=3$ and $d=1$. We take the Black-Scholes model, i.e. $\m_k(dx)=\rho_k(x)dx$ with
\b*
\rho_k(x) ~:=~ {\mathds 1}_{(0,\infty)}(x)\frac{\exp\big(-(\log(x)+2^{k-4})^2 \big/ 2^{k-2}\big)}{x\sqrt{2^{k-2}\pi}},~ \mbox{ for } k=1,2,3, 
\e*
and consider Lookback and Asian options, i.e.  $c(x,y,z):=\max(x,y,z)-z$ and $c(x,y,z):=\big((x+y+z)/3-\lambda z\big)^+$ with $\lambda \ge 0$. Notice that all $\rho_k$ are $L-$Lipschitz on $\R$ with $L=12$, and have finite $\theta-$moments for all $\theta>1$ with
\b*
\int_{\R}|x|^{\theta}\rho_k(x)dx ~\le ~ \int_{\R}|x|^{\theta}\rho_3(x)dx ~=~  e^{\theta(\theta -1)} ~=:~ M_{\theta}. 
\e*
As all $\m_k$ have unbounded support, we employ the full procedure of approximation in Section \ref{ssec:deter}. Let $\tmkn$ be supported on $\big\{i/n: 0\le i< mn \big\}$ , and defined by 
\b* 
\tmkn \big[\big\{i/n\big\}\big] ~:=~ \begin{cases}
1- \sum_{j=1}^{mn-1} \rho_k(x_{k,j})/n &  \mbox{ if }  i=0, \\
\rho_k(x_{k,i})/n &  \mbox{ if }  1\le i < mn,
\end{cases}  
\e* 
 where $x_{k,i}:={\rm argmin}_{x\in [i/n,(i+1)/n]} \rho_k(x)$ for $i=1,\ldots, mn-1$. Proposition \ref{prop:estimation} implies that
 \b*
 \Wc^{\oplus}\big((\tn_{1,m},\tn_{2,m},\tn_{3,m}),(\m_1,\m_2,\m_3)\big) ~\le~ 3\left(1/n  + M_{\theta}/m^{\theta-1} + \inf_{1\le j\le m} \Big\{2j^{2}L/n + 4M_{\theta}/j^{\theta-1}\Big\} \right).
 \e* 
Taking $j=m=m_n:= \lfloor \big(n(\theta-1)M_{\theta}/L\big)^{1/(\theta+1)}  \rfloor$ and setting $\m_k^n:=\tilde{\m}^{(n)}_{k,m_n}$, one has
 \b*
 \Wc^{\oplus}\big((\m_1^n,\m_2^n,\m_3^n),(\m_1,\m_2,\m_3)\big) ~\le~ 3\left(1/n  + M_{\theta}/m_n^{\theta-1} + 2m_n^{2}L/n + 4M_{\theta}/m_n^{\theta-1}\right)~:=~ \eps_n,
 \e*
where $\lim_{n\to\infty}\eps_n=0$. 
Numerical solutions to the \textsf{LP} problems for Lookback and Asian options with $\lambda=2$, corresponding to the above discretization, are presented in Figure \ref{fg:ex2_value}.
  \begin{figure}[h]
  \begin{minipage}[b]{0.47\textwidth}
\includegraphics[width=\textwidth,height=0.68\textwidth]{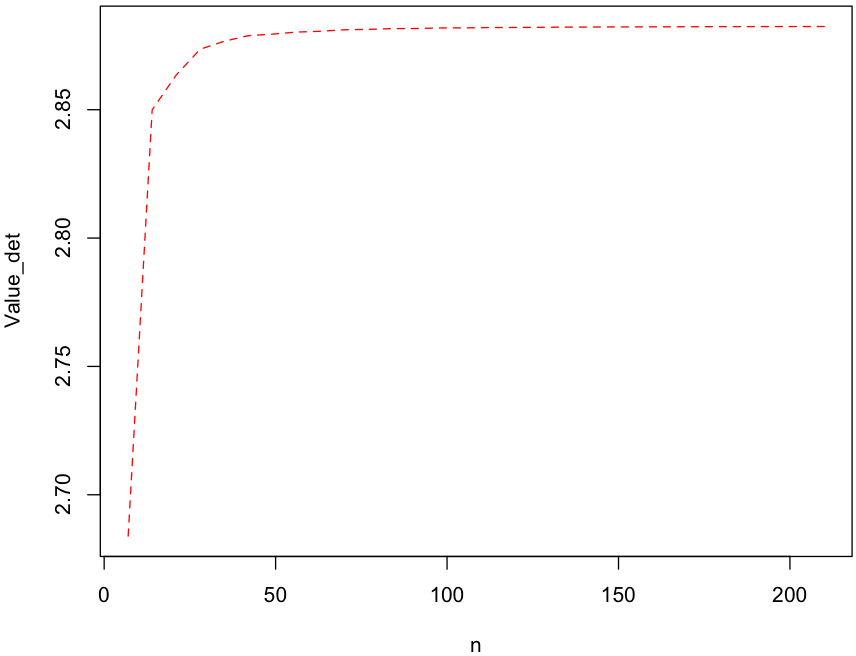}
  \end{minipage}
    \hfill
    \begin{minipage}[b]{0.47\textwidth}  \includegraphics[width=\textwidth,height=0.68\textwidth]{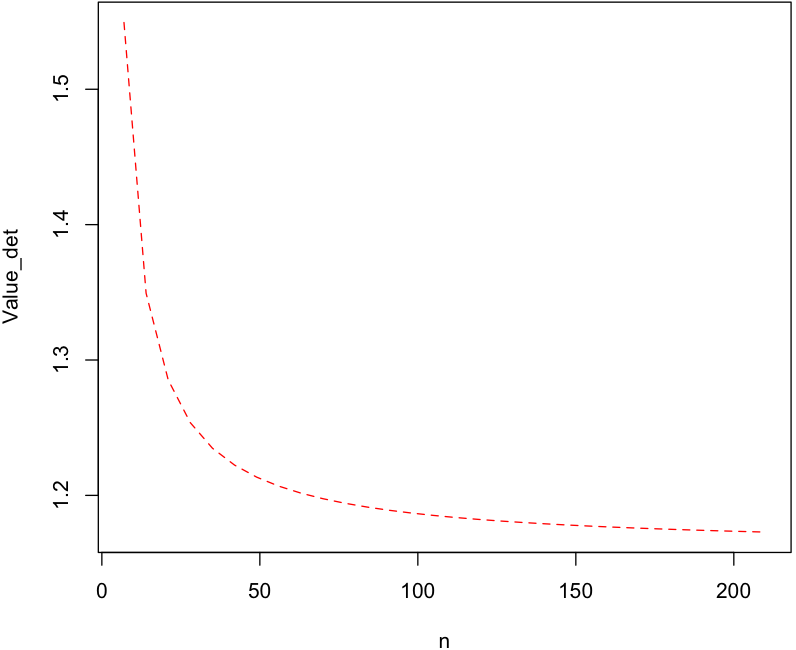}
      \end{minipage}
         \caption{Computations for Example \ref{ex2}. The two panes show the values $\pp_{\eps_n}(\m_1^n,\m_2^n,\m_3^n)$ for $10\le n\le 200$. The left pane stands for the Lookback option and the right one for the Asian option.}\label{fg:ex2_value}
\end{figure}
In Figure \ref{fg:ex2_heatmap}, we exhibit the heat maps of the optimizers projected respectively on $(S_1, S_2)$ and $(S_2,S_3)$ for $n=100$. 
\begin{figure}[h]
  \begin{minipage}[b]{0.47\textwidth}
\includegraphics[width=\textwidth,height=0.67\textwidth]{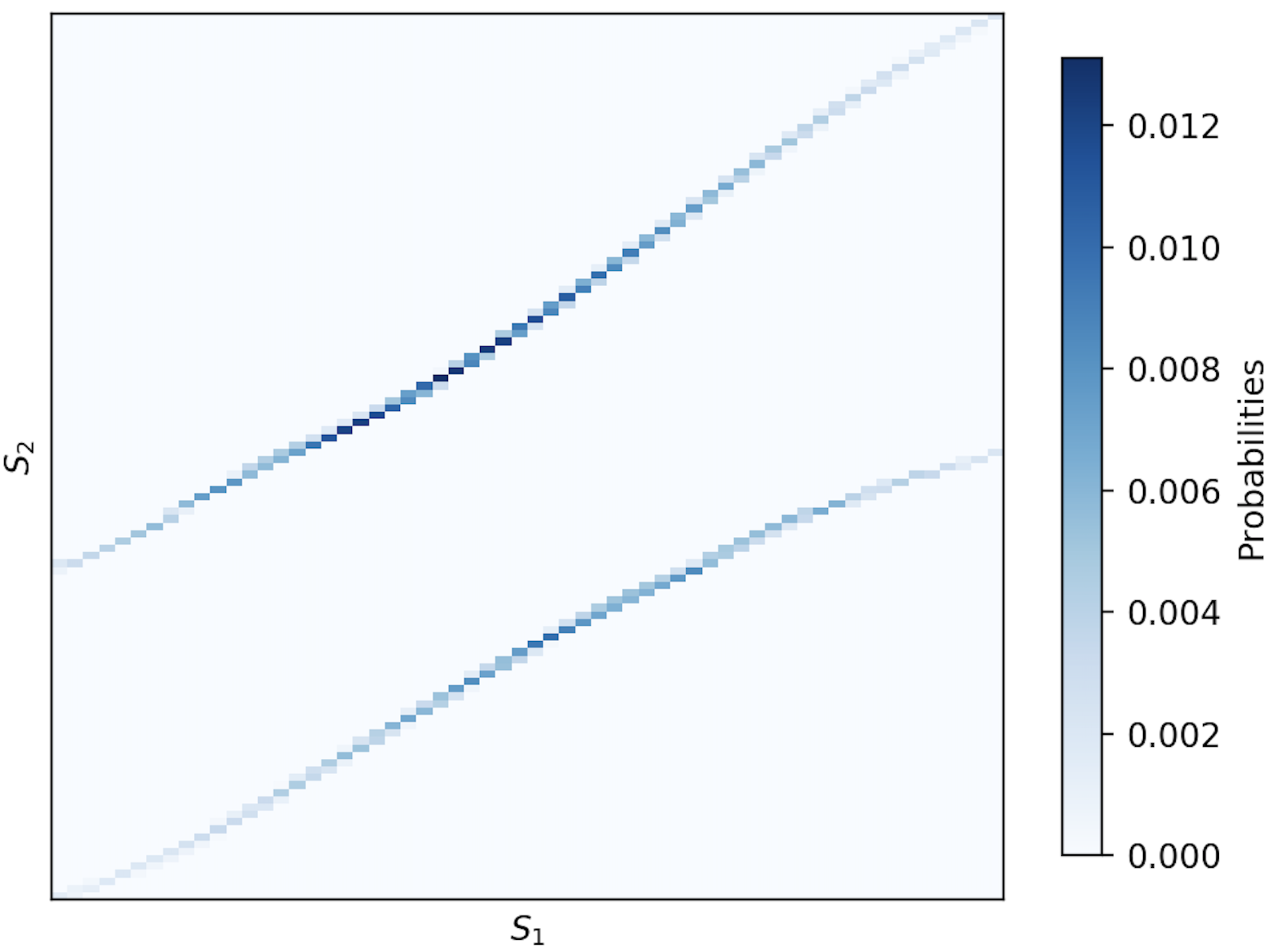}
  \end{minipage}
    \hfill
  \begin{minipage}[b]{0.47\textwidth}  \includegraphics[width=\textwidth,height=0.67\textwidth]{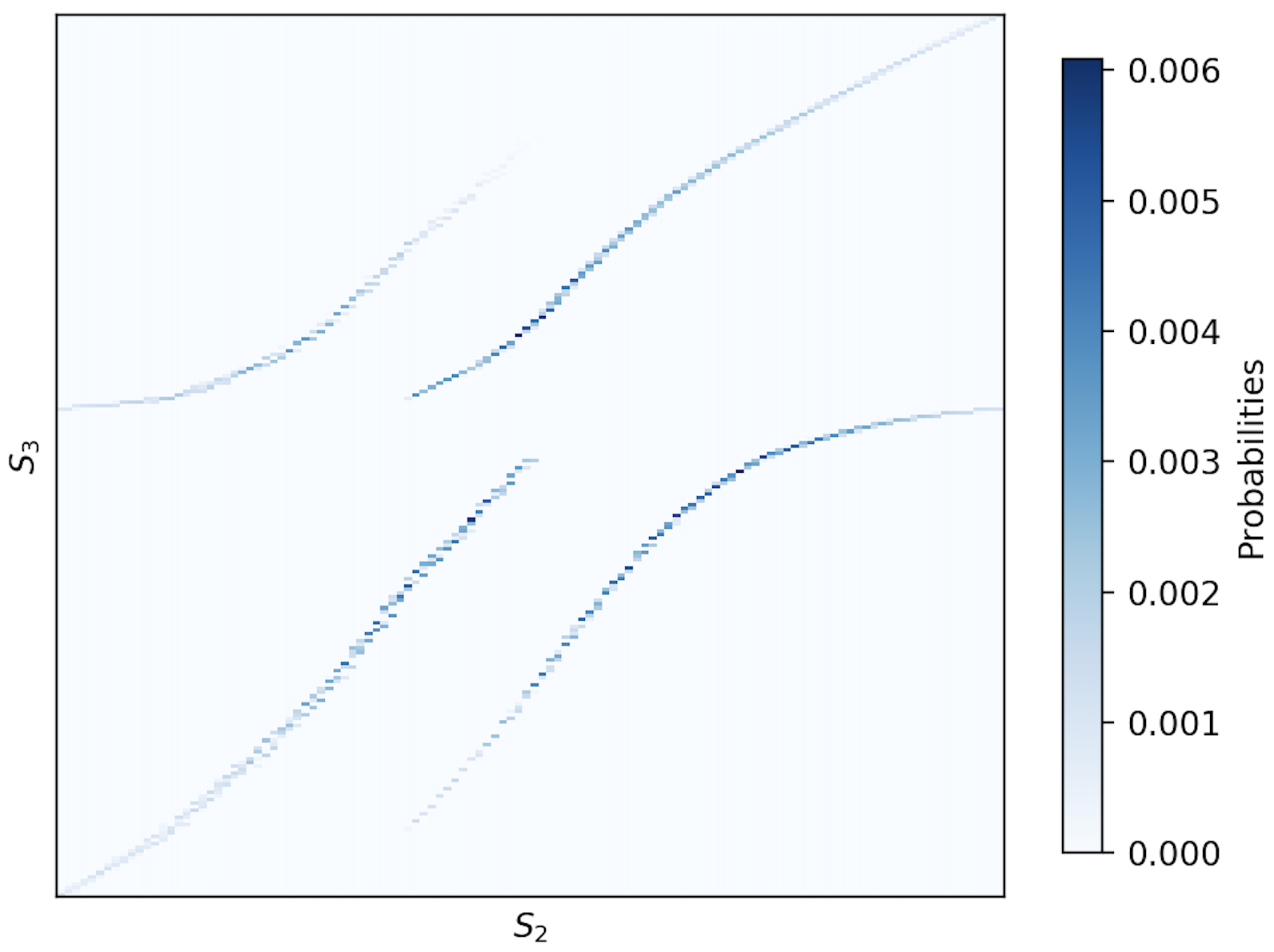}
  \end{minipage}
   \hfill
    \begin{minipage}[b]{0.47\textwidth}
\includegraphics[width=\textwidth,height=0.67\textwidth]{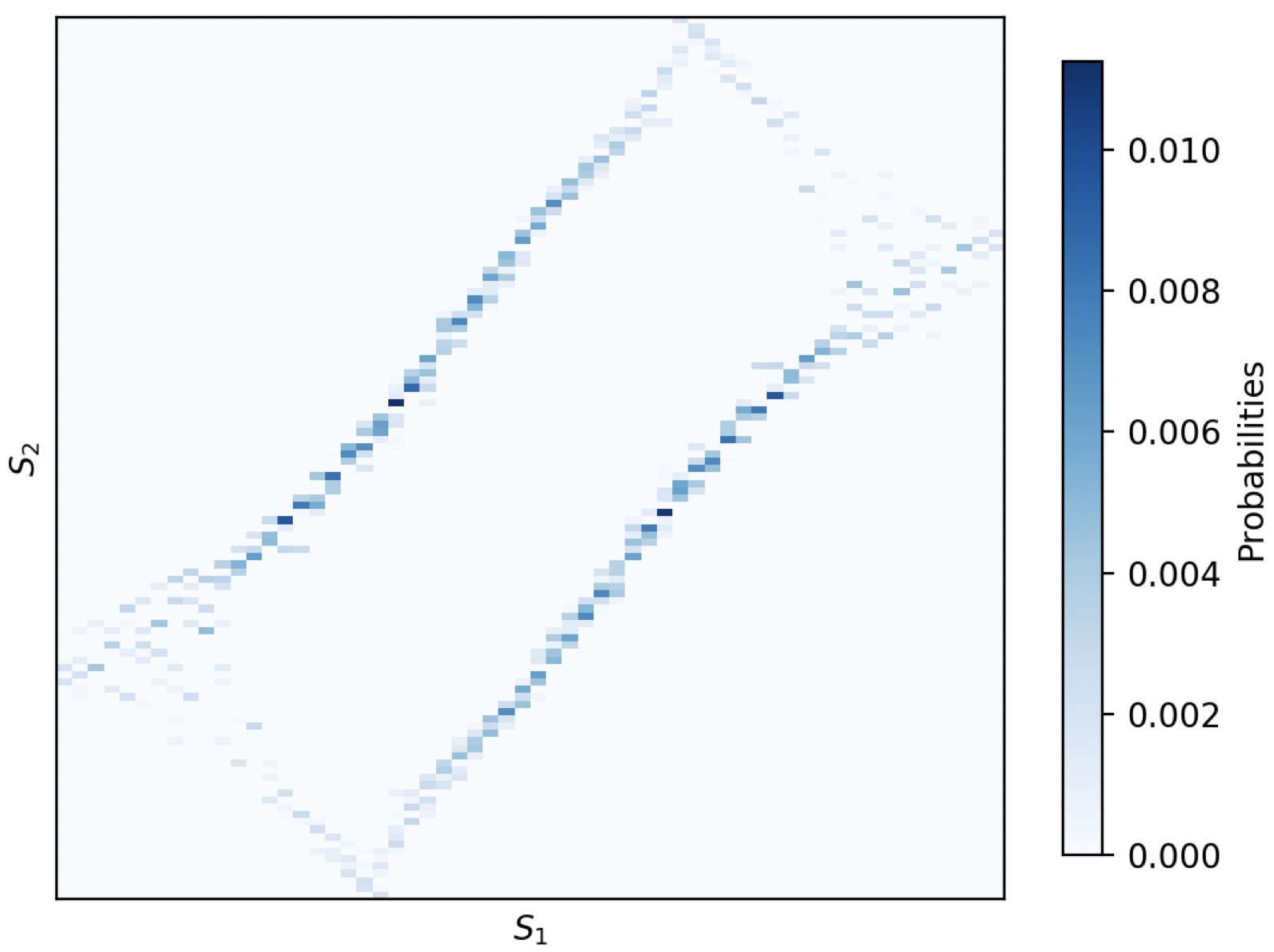}
  \end{minipage}
    \hfill
  \begin{minipage}[b]{0.47\textwidth}  \includegraphics[width=\textwidth,height=0.67\textwidth]{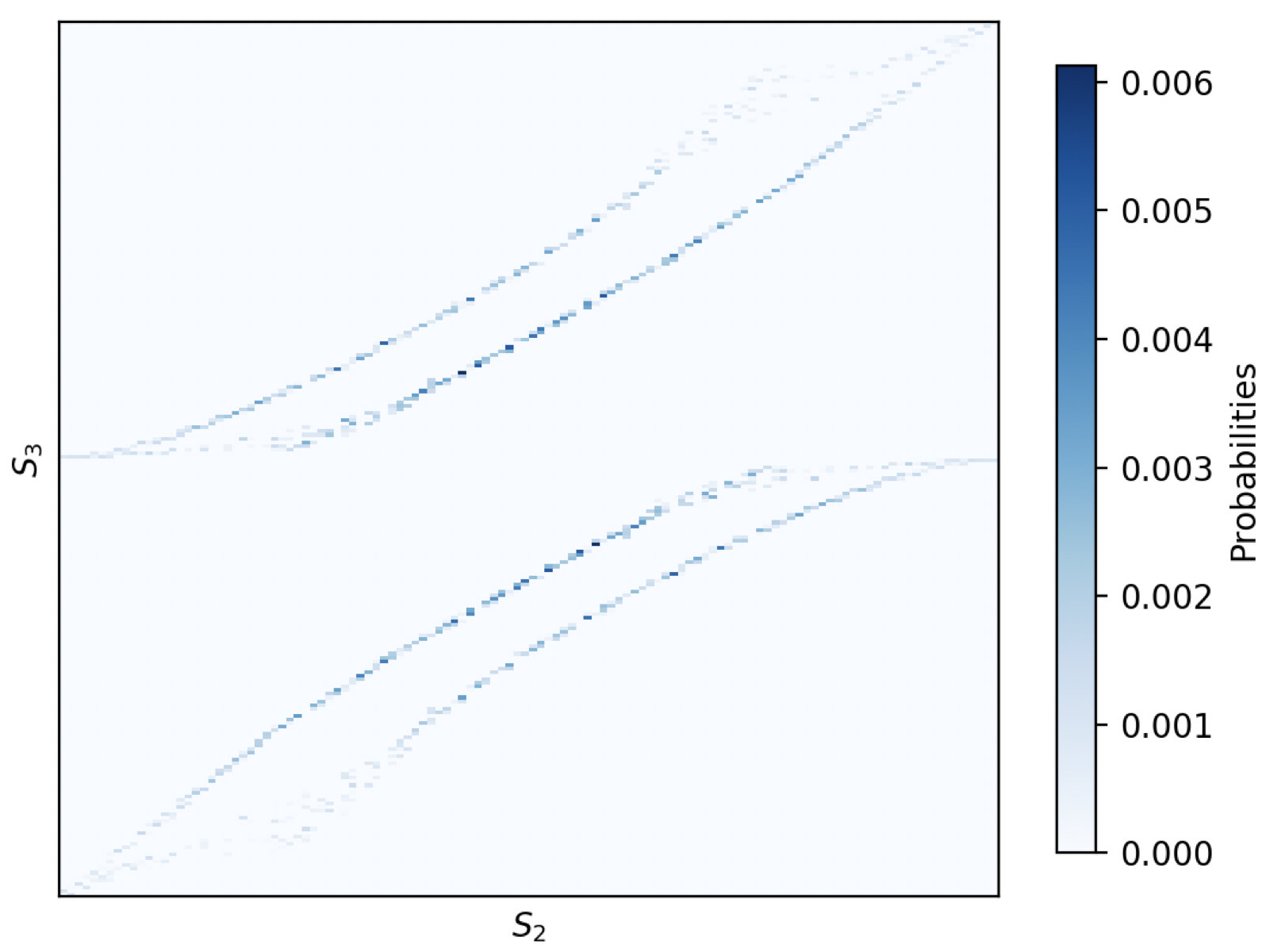}
  \end{minipage}
      \caption{Computations for Example \ref{ex2}. The four panes exhibit the heat maps of the optimizer projected on $(S_1,S_2)$ and $(S_2,S_3)$ for $n=100$. The top two correspond to the Lookback option and the bottom two correspond to the Asian option.}\label{fg:ex2_heatmap}
\end{figure}
The two panes above are for the Lookback option, where conditioning on $S_1$, $S_2$ takes two values, while conditioning on $S_2$, $S_3$ may take up to four values. The two panes below are for the Asian option. It appears that $(S_1, S_2)$ is concentrated on the boundary of a quadrilateral polygon, and $(S_2, S_3)$ is on the boundaries of two disjoint quadrilateral polygons.
\end{example}

\begin{example}\label{ex3}
Recently, the geometry of \textsf{MOT} problems in general dimensions has been studied, see e.g. \cite{LGK,OS, DMT}. We provide here numerical evidence in two dimensions which casts doubt on Conjecture 2 of \cite{LGK}. Take $N=d=2$ and $c(\bx,\by):=-\sqrt{(x_1-y_1)^2+(x_2-y_2)^2}$ for all  $\bx=(x_1,x_2)$, $\by=(y_1,y_2)\in\R^2$. Given $(\m,\n)\in\Pc^{\pq}$, the conjecture is stated as follows: If $\m$ admits a density on $\R^2$, then the support of $\P^*_{\bx}$ contains at most three points for $\m$ - a.e. $\bx\in\R^2$, where $\P^*\in\Mc(\m,\n)$ is the optimizer for $\pp(\m,\n)$ and $(\P^*_{\bx})_{\bx\in\R^2}$ is the regular conditional disintegration of $\P^*$ with respect to $\m$. Let $\m(d\bx)=\rho(\bx)d\bx$ and $\n(d\by)=\sigma(\by)d\by$ be identified by $\rho(\bx) := {\mathds 1}_{[-1,1]^2}(\bx)/4$ and 
\b*
 \sigma(\by) ~:=~ \frac{2-y_1}{4}{\mathds1}_{[1,2]\times [-1,1]}(\by) + \frac{2+y_1}{4}{\mathds 1}_{[-2,-1]\times [-1,1]} (\by) + \frac{2-y_2}{4}{\mathds 1}_{[-1,1]\times [1,2]} (\by) \\
  + \frac{2+y_2}{4}{\mathds 1}_{[-1,1]\times [-2,-1]} (\by).
\e*
Note that $\m$, $\n$ have bounded support and the deterministic discretization $\m^{(n)}$ and $\n^{(n)}$ of Section \ref{ssec:deter} may be computed explicitly. We obtain, $\m^{(n)}\big[\big\{(i/n,j/n)\big\}\big]=1/4n^2$ for  $i$, $j=-n\ldots, n-1$ and, for $i'$, $j'=-2n\ldots, 2n-1$,
\b*
 \n^{(n)}\big[\big\{(i'/n,j'/n)\big\}\big]~=~ \begin{cases}
(4n+2i'+1)/8n^3 &\mbox{if }  -2n\le i'\le -n-1,~ -n\le j' \le n-1, \\
(4n+2j'+1)/8n^3 &\mbox{if }  -n\le i'\le n-1,~ -2n\le j' \le -n-1, \\
(4n-2j'-1)/8n^3 &\mbox{if }  -n\le i'\le n-1,~ n\le j' \le 2n-1, \\
(4n-2i'-1)/8n^3 &\mbox{if }  n\le i'\le 2n-1,~ -n\le j' \le n-1, \\
0, &\mbox{otherwise}. 
\end{cases} 
\e*
With $\eps_n:=4/n \ge\Wc^{\oplus}\big((\m^{(n)},\n^{(n)}),(\m,\n)\big)$, we obtain the \textsf{LP} problem $\pp_{\eps_n}(\m^{(n)},\n^{(n)})$. For comparison, we also consider an approximation based on the random discretization. As in Example \ref{ex1}, we denote $\hm$, $\hn$ the corresponding  empirical measures and $\pp_{\eps_m}(\hm,\hn)$ the  \textsf{LP} problem. As $\n$ has a bounded support, with $C(3,2)$ defined in Proposition \ref{prop:sto_disc}, one has
\b*
\E\big[\Wc^{\oplus}\big((\hm,\hn),(\m,\n)\big)\big] ~\le~ 2C(3,2)(1+\log n)n^{-1/2}~=:~\chi(n).
\e*
Setting 
$\hnm:=\lfloor m^{r} \rfloor$ with $r=4.1$, one has  $\sum_{m\ge 1} \chi(\hnm)/\eps_m <\infty$ and $\lim_{m\to\infty}\pp_{\eps_m}(\hat{\m}^{\hnm}, \hat{\n}^{\hnm})=\pp(\m, \n)$ almost surely. Solving $\pp_{\eps_n}(\m^{(n)},\n^{(n)})$ and  $\pp_{\eps_m}(\hat{\m}^{\hnm}, \hat{\n}^{\hnm})$, the values are plotted in Figure \ref{fg:square}, where the convergence is illustrated.  
\begin{figure}[h]
 \begin{minipage}[b]{0.47\textwidth}
\includegraphics[width=\textwidth,height=0.68\textwidth]{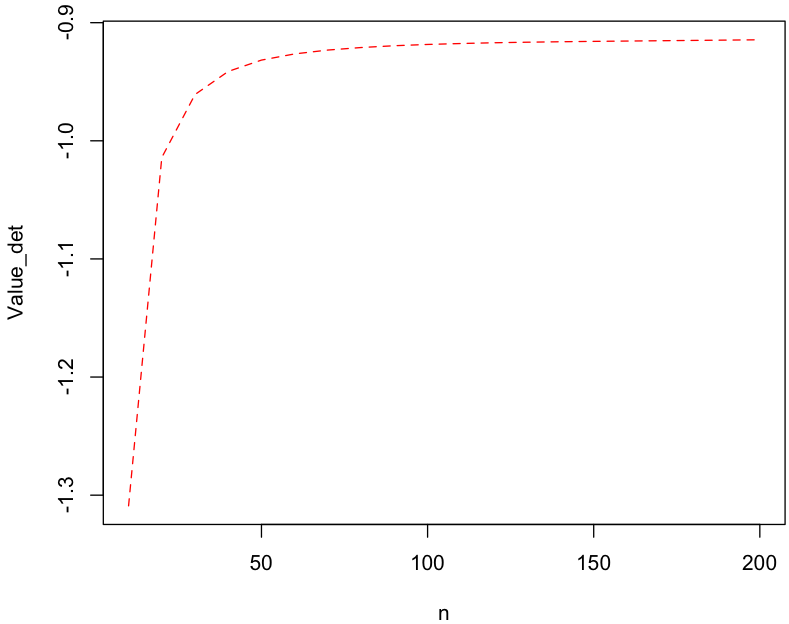}
  \end{minipage}
    \hfill
  \begin{minipage}[b]{0.47\textwidth}  \includegraphics[width=\textwidth,height=0.68\textwidth]{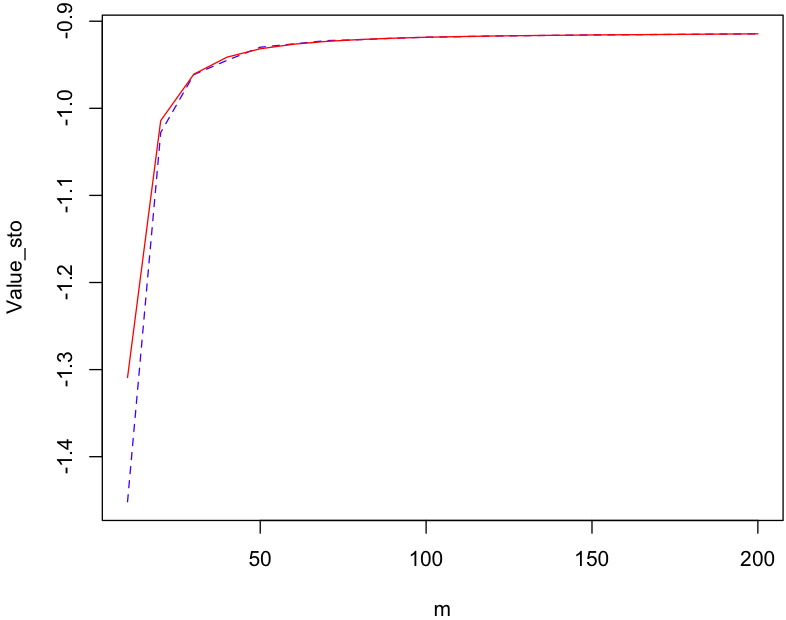}
  \end{minipage}
       \caption{Computations for Example \ref{ex3}. The left pane shows the values of $\pp_{\eps_n}(\m^{(n)},\n^{(n)})$ for $10\le n\le 200$ and the right pane shows the values $\pp_{\eps_m}(\hat{\m}^{\hnm}, \hat{\n}^{\hnm})$ (dashed line) for $10\le m\le 200$.}\label{fg:square}
\end{figure}
Note that, the complexity of $\pp_{\eps_m}(\hat{\m}^{\hnm}, \hat{\n}^{\hnm})$ is of order $m^{2r}$, which is as same as in Example \ref{ex1}, however the complexity of $\pp_{\eps_n}(\m^{(n)},\n^{(n)})$ is of order $n^4$ which is the square of that in the one-dimensional case. In Figure \ref{fg:square_heatmap}, we draw the heat map of the optimizers projected on $\big(S_1^{(1)},S_1^{(2)},S_2^{(1)}\big)$, where we recall that $\bS_1=\big(S_1^{(1)},S_1^{(2)}\big)$ and $\bS_2=\big(S_2^{(1)},S_2^{(2)}\big)$. As $\m$ and $\n$ are invariant by the map $\R^2\ni  (x,y)\mapsto (y,x)\in\R^2$,  $\big(S_1^{(1)},S_1^{(2)},S_2^{(1)}\big)$ and  $\big(S_1^{(2)},S_1^{(1)},S_2^{(2)}\big)$ are indistinguishable in law under the optimizer. The areas highlighted in red correspond to the values of $\bS_1$ which are transferred into more than three points. These clearly appear to have positive mass in disagreement  with Conjecture 2 in \cite{LGK}.
\begin{figure}[h]
   \begin{minipage}[b]{0.47\textwidth}  \includegraphics[width=\textwidth,height=0.7\textwidth]{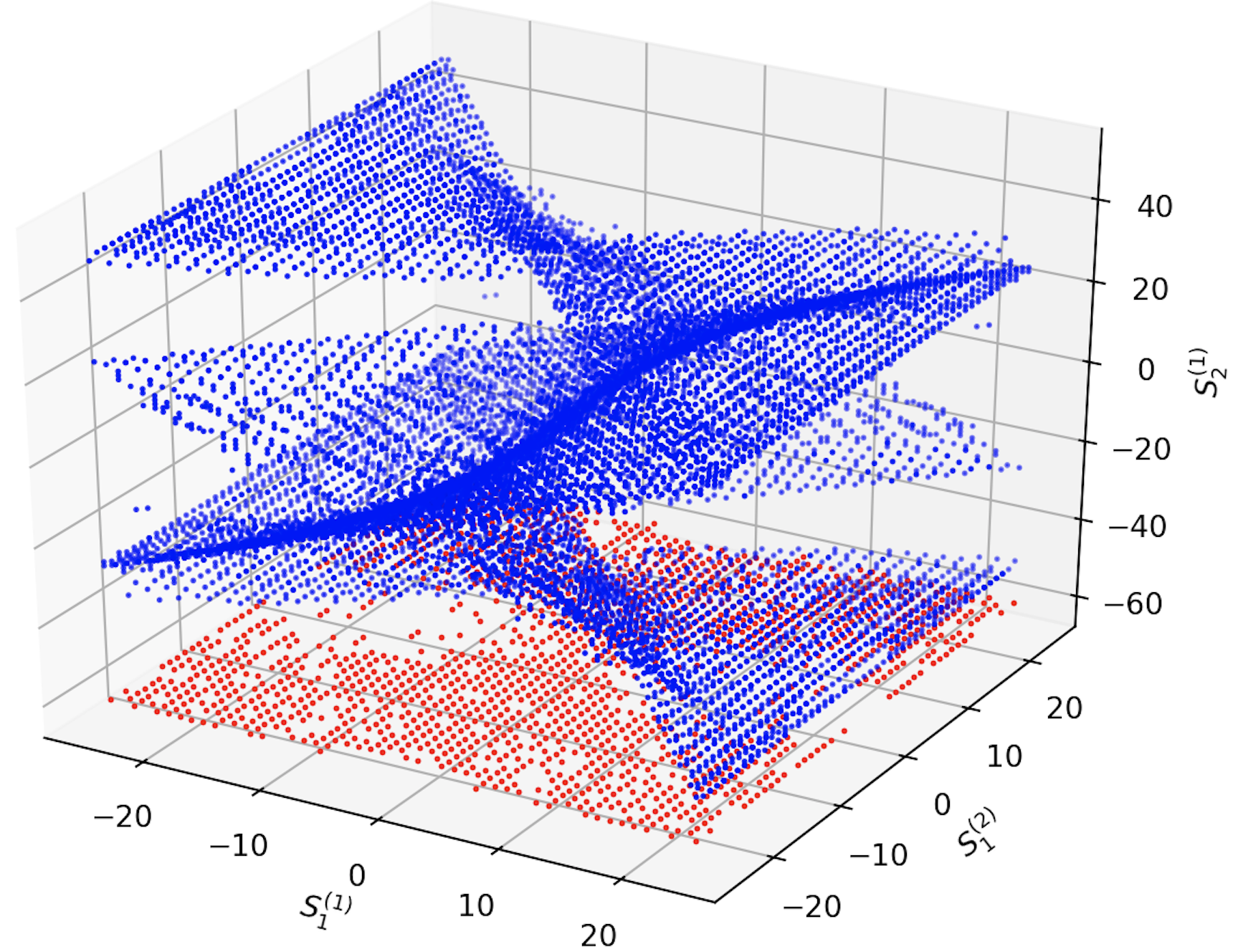}
  \end{minipage}
 \hfill
  \begin{minipage}[b]{0.47\textwidth}  \includegraphics[width=\textwidth,height=0.7\textwidth]{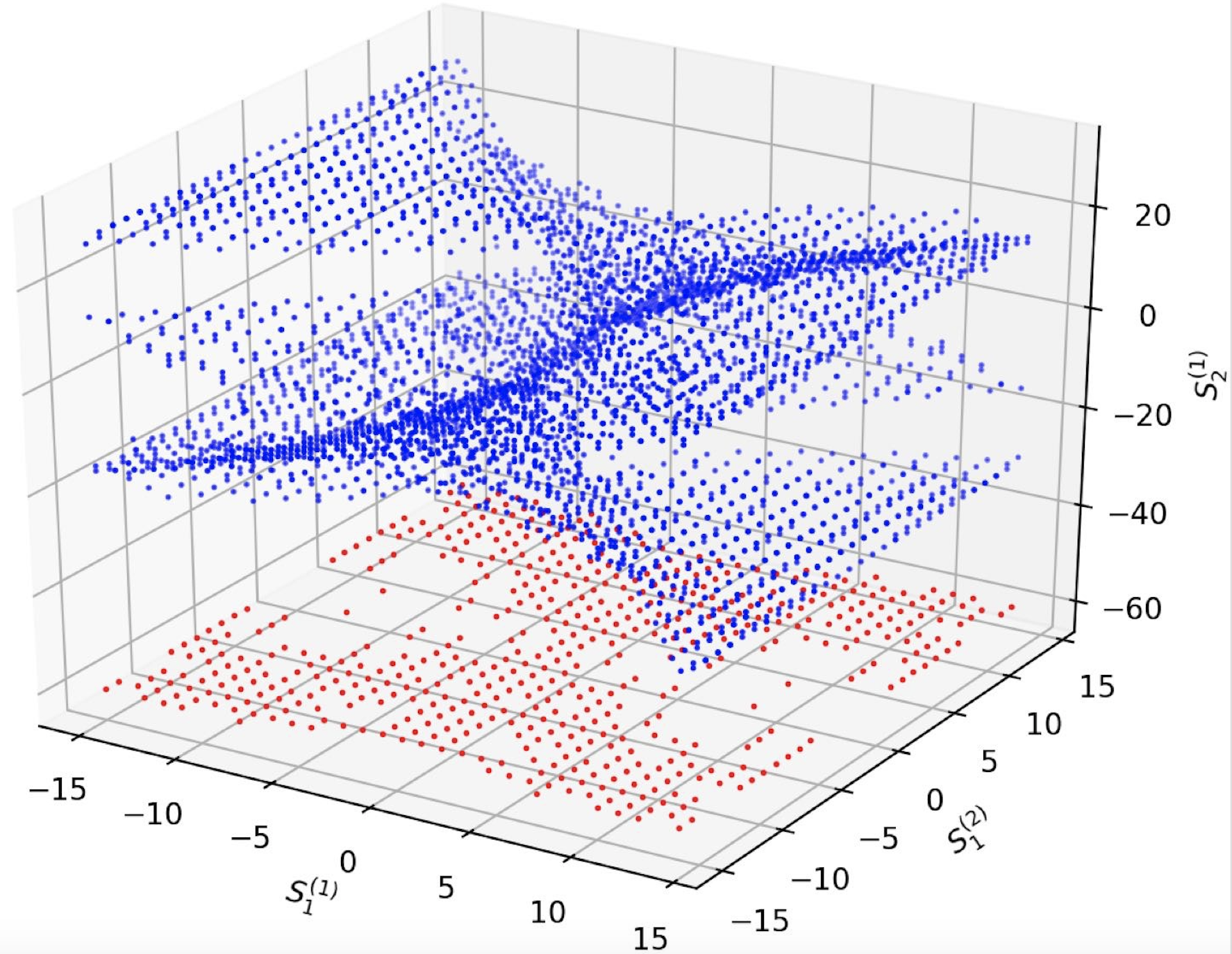}
  \end{minipage}
       \caption{Computations for Example \ref{ex3}. The left pane shows the heat map of the optimizer projected on $\big(S_1^{(1)},S_1^{(2)},S_2^{(1)}\big)$ for $n=100$ and the right pane shows  the heat map of the optimizers projected on $\big(S_1^{(1)},S_1^{(2)},S_2^{(1)}\big)$ for $m=100$.}\label{fg:square_heatmap}
\end{figure}
 \end{example}

\begin{example}\label{ex4}
To show the universality of our method, we consider in the last example an \textsf{MOT} problem in $\R^3$, i.e. $N=2$ and $d=3$. Let $c(\bx,\by):= \big(\sum_{i=1}^3 \lambda_i|x_i-y_i| -K\big)^+$ 
 for all $\bx=(x_1,x_2,x_3)$, $\by=(y_1,y_2,y_3)\in\R^3$, where $K>0$, $\lambda_i\ge 0$ and $\sum_{i=1}^3\lambda_i=1$. Here $c$ represents the payoff of a basket option written on three forward start options with strike $K$. We construct $(\m,\n)\in\Pc^{\preceq}$ in the following way. Let $\rho:\R^3\to [0,\infty)$ be an $L-$Lipschitz density function with a  finite $\theta-$moment for some $\theta>1$ and denote $\m(d\bx)=\rho(\bx)d\bx$. We define next $\n$ as the convolution of $\m$ with a standard normal distribution, i.e. $\n(d\by)=\sigma(\by)d\by$, where
\b*
\sigma(\by) ~:=~ \int_{\R^3}\rho(\by-\bx)\frac{1}{(2\pi)^{3/2}}\exp\left(-\frac{x_1^2+x_2^2+x_3^2}{2}\right)d\bx.
\e*  
Then it turns that $\sigma$ is $L-$Lipschitz and $\n$ admits finite $\theta-$moment. We are now under the same conditions as Example \ref{ex2}. Taking $\lambda_1=1/2$, $\lambda_1=1/3$, $\lambda_3=1/6$,  $K=1$ and 
\b*
\rho(\bx) ~:=~ {\mathds 1}_{[-1,1]^3}(\bx)\frac{|x_1| +|x_2x_3|}{C(1+x_1^2+2x_2^2+3x_3^2)},~ \mbox{ where } C~:=~ \int_{[-1,1]^3}\frac{|x_1|+|x_2x_3|}{1+x_1^2+2x_2^2+3x_3^2}d\bx,
\e*
one has $L=7/C$ and further 
\b*
\int_{\R^3}|\by|^2 \sigma(\by)d\by ~\le~ \frac{3}{C}\left(\frac{9}{2}+\frac{8}{\sqrt{2\pi}}\right) ~:=~M_2 &\mbox{and}& \chi(n)~:=~2C(2,3)n^{-1/3},
\e*
where $C(2,3)$ is given in Proposition \ref{prop:sto_disc}. We carry out the same discretization procedure, deterministic as in Example \ref{ex2} and random as in Examples \ref{ex1} and \ref{ex3}, and solve the corresponding \textsf{LP} problems. The resulting value functions are displayed in Figure \ref{fg:ex4}. 
 \begin{figure}[h]
  \begin{minipage}[b]{0.47\textwidth}
\includegraphics[width=\textwidth,height=0.68\textwidth]{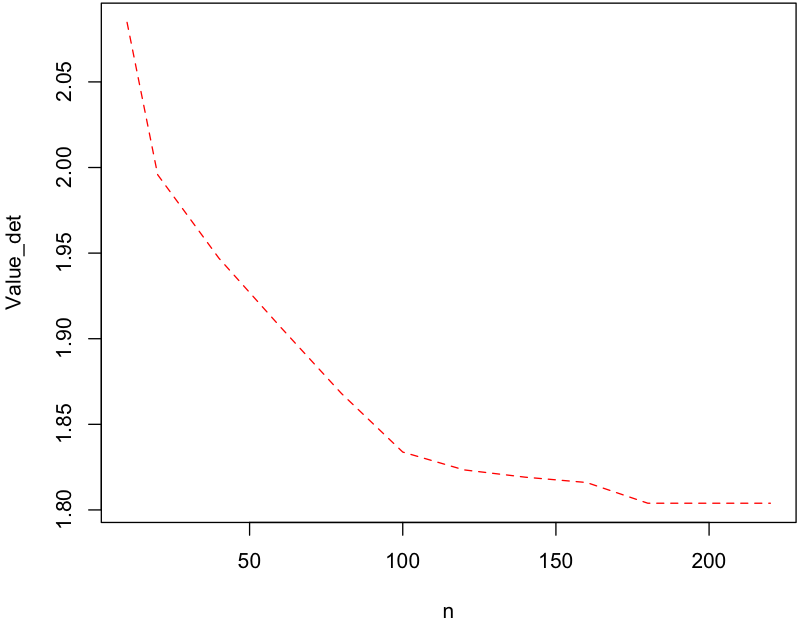}
  \end{minipage}
    \hfill
   \begin{minipage}[b]{0.47\textwidth}
\includegraphics[width=\textwidth,height=0.68\textwidth]{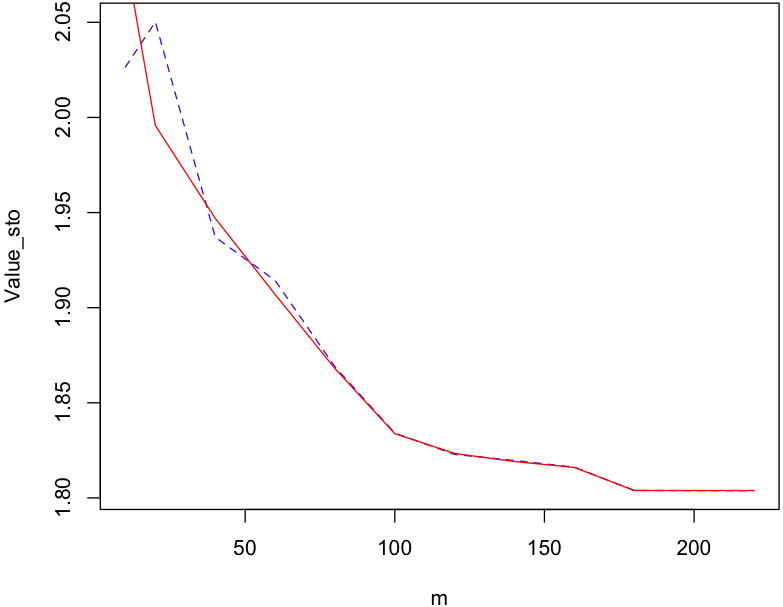}
  \end{minipage}
   \caption{Computations for Example \ref{ex4}. The left pane shows the values $\pp_{\eps_n}(\m^{n},\n^{n})$ for $10\le n\le 220$ and the right pane shows the values $\pp_{\eps_m}(\hat{\m}^{\hnm}, \hat{\n}^{\hnm})$ (dashed line) for $10\le m\le 220$.}\label{fg:ex4}
\end{figure}
 \end{example}

\section{Proofs}\label{sec:proof}

Section \ref{sec:proof} is devoted to the proofs of Theorems \ref{thm:general} and \ref{thm:convergence_rate}. Similar to the usual \textsf{MOT}, the relaxed problem $\pp_{\eps}(\bm)$ admits a dual formulation given by
\begin{eqnarray}\label{def:dp}
\dd_{\eps}(\bm) ~:=~ \inf_{(H,\ps)\in\Dc_{\eps}}~ \left[\sum_{k=1}^N\int \ps_k d\m_k\right],
\end{eqnarray}
where $\Hc$ is the set of $\F-$adapted processes $H=(H_k)_{1\le k\le N-1}$ taking values in $\R^d$, i.e. $H_k\in\L^{\infty}(\Om^k;\R^d)$ for $k=1,\ldots, N-1$, and $\Dc_{\eps}\subset\Hc\times \Lam^N$ denotes the collection of pairs $\big(H=(H_k)_{1\le k\le N-1},\ps=(\ps_k)_{1\le k\le N}\big)$ such that for $(\bx_1,\ldots, \bx_N)\in\Om^N$
\be\label{ineq:dual}
 \sum_{k=1}^{N-1} H_k(\bx_1,\ldots, \bx_k)\cdot (\bx_{k+1}-\bx_k) - \eps \sum_{k=1}^{N-1}\|H_k\|_{\infty} + \sum_{k=1}^N\ps_k(\bx_k) ~\ge~ c(\bx_1,\ldots, \bx_N).
\ee
Recall that $\Mc_{\eps}(\bm)\subset \Pc(\bm)$ is convex and compact. An application of the Min-Max theorem allows to establish the  Kantorovich duality between \eqref{def:rmotp} and \eqref{def:dp} in Theorem \ref{thm:duality} below. The proof largely repeats the reasoning in \cite{BHLP} where the result was shown for $\eps=0$, but it is nonetheless included in Appendix \ref{sec:append}. 

\begin{theorem}\label{thm:duality}
Let $\bm\in\Pc_{\eps}^{\pq}$. If $c$ is upper semicontinuous with linear growth, then there exists an optimiser $\P^*$ for $\pp_{\eps}(\bm)$, i.e. $\P^*\in \Mc_{\eps}(\bm)$ and $\pp_{\eps}(\bm)=\E_{\P^*} [c]$. Moreover, there is no  duality gap, i.e. $\pp_{\eps}(\bm)=\dd_{\eps}(\bm)$.
\end{theorem}

For $\eps=0$, the left-hand side of \eqref{ineq:dual}
 stands for a super-replication of the payoff $c$ by trading dynamically in the underlying assets and statically in a range of Vanilla options. More precisely, $H_k(\bS_1,\ldots, \bS_k)$ denotes the number of shares held by the trader at time $k$. Vanilla options allow the holder to receive the cash flow equal to $\ps_k(\bS_k)$ at time $k$ for $k=1,\ldots, N$, and their market price is given as the integral of $\ps_k$ with respect to $\m_k$, where $\m_1,\ldots, \m_N$ represent the market-implied distributions of $\bS_1,\ldots, \bS_N$. When $d=1$,  as observed by Breeden and Litzenberger \cite{BL}, $\m_k$ are uniquely determined from the observed prices of call/put options for all possible strikes. In consequence, the expression in brackets on the right-hand side of \eqref{def:dp} represents the cost of pursuing a superhedging strategy $(H,\psi)$ and $\dd(\bm)$ is equal to the minimal superhedging price of $c$.
 
\subsection{Convergence of relaxed \textsf{MOT} problems}\label{ssec:general}

Theorem \ref{thm:general} shows that $\pp(\bm)$ can be approximated by considering a sequence of relaxed \textsf{MOT} problems, which provides the main insight into our proposed scheme for solving  \textsf{MOT} problems. The proof of Theorem \ref{thm:general} is divided into the proofs of  Corollary \ref{cor1}, Proposition \ref{prop:rusc} and Lemma \ref{lem:compact}. 

\begin{proposition}\label{prop:approx}
Let $\bm\in\Pc_{\eps}^{\pq}$. Then for any $\bn\in \Pc^N$, one has $\bn\in \Pc_{\eps+r}^{\pq}$ with  $r:=\Wc^{\oplus}(\bm, \bn)$. If we assume in addition that $c$ is Lipschitz with Lipschitz constant ${\li}(c)$, then  $
\pp_{\eps}(\bm)\le \pp_{\eps+r}(\bn)+{\li}(c)r$. \end{proposition}

\begin{proof}
Set $r_k:=\Wc(\m_k,\n_k)$ for $k=1,\ldots, N$ and one has by definition $r = \sum_{k=1}^N r_k$. Take an arbitrary $\P\in \Mc_{\eps}(\bm)$. It follows from  Theorem 1 of  Skorokhod \cite{Skorokhod} that, there exists an enlarged probability space $(E , \Ec , \Qc)$ which supports random variables $X_k$  and $Z_k$ taking values in $\R^d$ for $k=1,\ldots, N$ such that
\b*
&&\Qc\circ (X_1, \ldots, X_N)^{-1} ~=~ \P, \\
&& Z_1, \ldots, Z_N \mbox{ and } (X_1,\ldots, X_N) \mbox{ are mutually independent}, \\
&& \Qc\circ Z_k^{-1} \mbox{ is a standard normal distribution on } \R^d,\q \mbox{for } k=1,\ldots, N.
\e*
For $k=1,\ldots, N$, let $\P_k$ be the optimal transport plan realizing the Wasserstein distance between $\m_k$ and $\n_k$, i.e. $\P_k\in\Pc(\m_k,\n_k)$ and $r_k=\E_{\P_k}\big[\big|\bS_1-\bS_2\big|\big]$. 
From Lemma \ref{lem:Skorokhod}, there exist measurable functions $f_k: \Om^2\to\R^d$ such that $\Qc\circ (X_k, Y_k)^{-1}= \P_k$ with  $Y_k:= f_k(X_k, Z_k)$, for $k=1, \ldots, N$, which yields in particular $\Qc\circ Y_k^{-1}=\n_k$. Furthermore, one has for all $h\in\Cc_{b}\big(\Om^k;\R^d\big)$
\b*
&&\E_{\Qc}\big[h(Y_1,\ldots, Y_k)\cdot(Y_{k+1}-Y_k)\big] \\
&= & \E_{\Qc}\big[h(Y_1,\ldots, Y_k)\cdot (Y_{k+1}-X_{k+1})\big] + \E_{\Qc}\big[h(Y_1,\ldots, Y_k)\cdot (X_{k+1}-X_k)\big] \\
&&+ \E_{\Qc}\big[h(Y_1,\ldots, Y_k)\cdot (X_k-Y_k)\big] \\
&\le & (r_k+r_{k+1})\|h\|_{\infty} + \E_{\Qc}\Big[h\big(f_1(X_1,Z_1),\ldots, f_k(X_k,Z_k)\big)\cdot \big(X_{k+1}-X_k\big)\Big] \\
&= &(r_k+r_{k+1})\|h\|_{\infty} + \int_{\Om^k} \E_{\P}\big[h\big(f_1(\bS_1,\bx_1),\ldots, f_k(\bS_k,\bx_k)\big)\cdot (\bS_{k+1}-\bS_k)\big]\Nc_k(d\bx_1,\ldots, d\bx_k) \\
&\le&(\eps+r)\|h\|_{\infty}, 
\e*
where $\Nc_k$ denotes the joint distribution of $ Z_1, \ldots, Z_k$. Therefore, $\E_{\P'}\big[h(\bS_1,\ldots, \bS_k)\cdot(\bS_{k+1}-\bS_k)\big]\le (\eps+r)\|h\|_{\infty}$ holds for all $h\in\Cc_b\big(\Om^k;\R^d\big)$, where $\P':=\Qc\circ (Y_1,\ldots, Y_N)^{-1}$. In view of the monotone class theorem, this is equivalent to  
\b*
\E_{\P'}\Big[\Big|\E_{\P'}\big[\bS_{k+1}\big|\Fc_k\big] -\bS_k\Big |\Big] ~ \le~ \eps+r.
\e*
Hence, $\P'\in \Mc_{\eps+r}(\bn)$ and  $\bn\in \Pc_{\eps+r}^{\pq}$. To conclude the proof, notice that 
\b*
&&\E_{\P}\big[c(\bS_1,\ldots, \bS_N)\big]-\pp_{\eps+r}(\bn) 
~\le~  \E_{\P}\big[c(\bS_1,\ldots, \bS_N)\big]-\E_{\P'}\big[c(\bS_1,\ldots, \bS_N)\big] \\
&=& \E_{\Qc}\big[c(X_1,\ldots, X_N)-c(Y_1,\ldots, Y_N)\big] ~\le~ {\li}(c)\sum_{k=1}^N \E_{\Qc}\big[\big|X_k-Y_k\big|\big] 
~=~ {\li}(c) r,
\e*
which yields $\pp_{\eps}(\bm)\le \pp_{\eps+r}(\bn)+{\li}(c) r$ since $\P\in \Mc_{\eps}(\bm)$ is arbitrary.
\end{proof}

In consequence, the corollary below follows immediately.

\begin{corollary}\label{cor1}
Let $(\bm^n)_{n\ge 1}$ and $(\eps_n)_{n\ge 1}$ be the   sequences in Theorem \ref{thm:general}. Then 
\b*
 \pp(\bm)~ \le~ \pp_{\eps_n}(\bm^n) + {\li}(c)\eps_n 
  ~\le~  \pp_{2\eps_n}(\bm) + 2{\li}(c)\eps_n,~ \mbox{ for all } n\ge 1.
\e*
\end{corollary}

\begin{proof}
Taking $\eps=0$, $\bn=\bm^n$ and $r=r_n$, one has $\pp(\bm)\le \pp_{r_n}(\bm^n)+{\li}(c)r_n\le \pp_{\eps_n}(\bm^n)+{\li}(c)\eps_n$, where $\pp_{r_n}(\bm^n)\le \pp_{\eps_n}(\bm^n)$ follows by definition. The second inequality follows with the same arguments but interchanging $\bm$ and $\bm^n$. 
\end{proof}

To complete the proof of Theorem \ref{thm:general}, it remains to show $\pp_{2\eps_n}(\bm)\to \pp(\bm)$ as $n\to\infty$. 

\begin{proposition}\label{prop:rusc}
Let $c$ be Lipschitz.

\vspace{1mm}

\no \rmi For every fixed $\eps\ge 0$, the map $
\Pc_{\eps}^{\pq} \ni \bm \mapsto   \pp_{\eps}(\bm)\in \R$
is upper semicontinuous under $\Wc^{\oplus}$.

\vspace{1mm}

\no \rmii For every fixed $\bm\in\Pc^{\pq}$, the map $
[0,\infty)\ni\eps\mapsto \pp_{\eps}(\bm)\in \R$ 
is non-decreasing, continuous and concave.
\end{proposition}

Before proving Proposition \ref{prop:rusc}, let us remark that, together with Corollary \ref{cor1} and Lemma \ref{lem:compact}, it yields an instant proof of our main result:
 
\begin{proof}[Proof of Theorem \ref{thm:general}]
\rmi We have $\Mc_{\eps_n}(\bm^n)\neq\emptyset$ from Proposition \ref{prop:approx}. Corollary \ref{cor1} yields $-{\li}(c)\eps_n \le  \pp_{\eps_n}(\bm^n)-\pp(\bm) \le   \big(\pp_{2\eps_n}(\bm)-\pp(\bm)\big)+{\li}(c)\eps_n$ for all $n\ge 1$, and Proposition \ref{prop:rusc} gives $\lim_{n\to\infty} \pp_{\eps_n}(\bm^n)=\pp(\bm)$.

\vspace{1mm}

\no\rmii  By Theorem \ref{thm:duality}, we know the existence of optimizer $\P_n$ for all $n\ge 1$. As $\P_n\in\Mc_{\eps_n}(\bm^n)\subset\Pc(\bm^n)$, it follows from Lemma \ref{lem:compact} that, $(\P_n)_{n\ge 1}$ is tight and every limit point must belong to $\Pc(\bm)$. Take an arbitrary convergent subsequence $(\P_{n_k})_{k\ge 1}$ with limit $\P$, then $\P\in\Pc(\bm)$. Repeating the proof of Proposition \ref{prop:rusc} {\rm (i)}, one deduces that $\P\in\Mc(\bm)$ and $\P$ is thus an optimizer for $\pp(\bm)$. If $\pp(\bm)$ has a unique optimizer, then every convergent subsequence of $(\P_n)_{n\ge 1}$ has the same limit, which shows that $(\P_n)_{n\ge 1}$ converges weakly as $\Pc(\Om^N)$ is Polish.
\end{proof}

\begin{proof}[Proof of Proposition \ref{prop:rusc}]
 \rmi We establish a slightly stronger property. Take two sequences $(\eps_n)_{n\ge 1}\subset [0,\infty)$ and $(\bm^n)_{n\ge 1}\subset\Pc_{\eps_n}^{\pq}$ with limits $\eps$ and $\bm$. Let $\P_n\in \Mc_{\eps_n}(\bm^n)$ satisfy $
 \limsup_{n\to\infty}\pp_{\eps_n}(\bm^n) =  \limsup_{n\to\infty}\E_{\P_n}\big[c(\bS_1,\ldots, \bS_N)\big]$. 
 Up to passing to a subsequence, we may assume that $\limsup_{n\to\infty}\E_{\P_n}[c]=\lim_{n\to\infty}\E_{\P_n}[c]$, and further by Lemma \ref{lem:compact} that $(\P_n)_{n\ge 1}$ converges in Wasserstein sense to some limit $\P\in\Pc(\bm)$. For every  $k=1,\ldots, N-1$ and $h\in \Cc_{b}(\Om^k;\R^d)$, one has $\E_{\P_n}\big[h(\bS_1,\ldots, \bS_k)\cdot (\bS_{k+1}-\bS_k)\big] \le \eps_n\|h\|_{\infty}$ for all $n\ge 1$ and hence for the limiting measures, which implies $\P\in \Mc_{\eps}(\bm)$.
 Similarly, Lipschitz continuity of $c$ gives
 \b*
\limsup_{n\to\infty} \pp_{\eps_n}(\bm^n)~=~ \lim_{n\to\infty}\E_{\P_n}\big[c(\bS_1,\ldots, \bS_N)\big] ~= ~ \E_{\P}\big[c(\bS_1,\ldots, \bS_N)\big] ~\le~ \pp_{\eps}(\bm).
 \e*
 \rmii We first prove the concavity. Given $\eps$, $\eps'\ge 0$ and $\alpha\in [0,1]$, it remains to show $(1-\alpha)\pp_{\eps}(\bm)+\alpha \pp_{\eps'}(\bm) \le \pp_{\eps_{\alpha}}(\bm)$,  where $\eps_{\alpha}:=(1-\alpha)\eps+\alpha\eps'$. This indeed follows from the fact that $(1-\alpha)\P+\alpha\P'\in \Mc_{\eps_{\alpha}}(\bm)$ for all $\P\in \Mc_{\eps}(\bm)$ and $\P'\in \Mc_{\eps'}(\bm)$. In particular, the map restricted to $(0,\infty)$ is continuous. Finally, the reasoning in {\rm (i)} above, with $\bm^n=\bm$ and $\eps_n\to 0$, gives $\lim_{n\to\infty} \pp_{\eps_n}(\bm)=\pp(\bm)$ which combined with the obvious reverse inequality yields the right continuity at $\eps=0$.
\end{proof}

\begin{lemma}\label{lem:compact}
Let $(\bm^n)_{n\ge 1}\subset\Pc^{N}$ be a sequence converging to $\bm\in\Pc^N$ under $\Wc^{\oplus}$, and $\P_n\in \Pc(\bm^n)$ for all $n\ge 1$. Then there exists a subsequence $(\P_{n_k})_{k\ge 1}$ converging in Wasserstein metric on $\Pc(\Omega^N)$ and its limit $\P$ belongs to $\Pc(\bm)$.
\end{lemma}

\begin{proof}
Taking the compact $E_R:=\big\{(\bx_1,\ldots, \bx_N)\in \Om^N: |\bx_k|\le R, \mbox{ for } k=1,\ldots, N\big\}$. One has
\b*
\P_n[E_R^c] ~ \le~ \sum_{k=1}^N \int_{\R^d} {\mathds 1}_{[R,\infty)}(|\bx|) \mu_k^n(d\bx) ~ \le ~ \frac{1}{R}  \sup_{n\ge 1}~ \left\{\sum_{k=1}^N \int_{\R^d} |\bx|\mu^n_k(d\bx)\right\}.
\e*
Further, the convergence under $\Wc^{\oplus}$ implies that 
\b*
\lim_{n\to\infty} \int_{\R^d} |\bx|\mu^n_k(d\bx) ~=~ \int_{\R^d} |\bx|\mu_k(d\bx),~ \mbox{ for } k=1,\ldots, N,
\e*
which yields the tightness of $(\P_n)_{n\ge 1}$ since $\lim_{R\to\infty}\big\{\sup_{n\ge 1}\P_n[E_R^c]\big\} = 0$. This implies there exists a weakly convergent subsequence, which we still denote $(\P_n)_{n\ge 1}$ and let $\P$ be its limit. Notice that the projection map $\bS_k$ is continuous, then $\mu^n_{k}=\P_{n}\circ \bS_k^{-1}$ also converges weakly to $\P\circ \bS_k^{-1}$ for $k=1,\ldots, N$, which shows $\P\in\Pc(\bm)$. Finally, the convergence of $\P_n$ to $\P$ in the Wasserstein sense follows from the convergence of first moment.
 \end{proof}

\subsection{Convergence rate analysis: $N=2$ and $d=1$}\label{ssec:convergence}

This section concerns the estimation of the convergence rate for the one-step model in dimension one. The duality in Theorem \ref{thm:duality} plays an important role and is used repeatedly. To the best of our knowledge, the error bound in Theorem \ref{thm:convergence_rate} is the first of its kind in the literature. Fix a pair $(\m,\n)\in\Pc^{\pq}$. Throughout Section \ref{ssec:convergence}, we stress the dependence on $c$ and write $\pp^c_{\eps}(\m,\n)\equiv \pp_{\eps}(\m,\n)$, $\Dc_{\eps}^{c}\equiv\Dc_{\eps}$ and  $\dd^c_{\eps}(\m,\n)\equiv \dd_{\eps}(\m,\n)$. Clearly, for any $c_1$ and $c_2:\R^2\to\R$,  it holds that $
\pp_{\eps}^{c_1+c_2}(\m,\n) \le   \pp_{\eps}^{c_1}(\m,\n)+ \pp_{\eps}^{c_2}(\m,\n)$. In view of Corollary \ref{cor1}, one has
\b*
\big|\pp^c_{\eps_n}(\m^n,\n^n)-\pp^c(\m,\n)\big| ~\le ~  \big(\pp^c_{2\eps_n}(\m,\n)-\pp^c(\m,\n)\big) + {\li}(c)\eps_n,
\e*
where $\eps_n\ge \Wc^{\oplus}\big((\m^n,\n^n),(\m,\n)\big)$ for all $n\ge 1$. For the purpose of estimating the difference $\big|\pp^c_{\eps_n}(\m^n,\n^n)-\pp^c(\m,\n)\big|$, we need to understand the asymptotic behavior of $\pp^c_{\eps}(\m,\n)-\pp^c(\m,\n)$ as $\eps$ goes to $0$, which is shown in the proof of Theorem \ref{thm:convergence_rate}.

\begin{proof}[Proof of Theorem \ref{thm:convergence_rate}]
Set $L:=\max\big({\li}(c), \sup_{(x,y)\in\R^2}\big|\partial^2_{yy}c(x,y)\big|\big)<\infty$ and introduce $c_L(x,y):=c(x,y)-Ly^2/2$. Then, for each $x\in\R$, the map $y\mapsto c_L(x, y)$ is concave. Further, let us truncate $c_L$ by an affine function with respect to $y$. Namely, define for every $R\ge 0$ 
\b*
c_L^R(x,y)~:=~ \begin{cases}
   c_L(x,-R) +  (y+R)\partial_y c_L(x,-R) &\mbox{if } y\le -R,\\
      c_L(x,y)   &\mbox{if } -R<y\le R,\\
  c_L(x,R) + (y-R)\partial_y c_L(x,R) &\mbox{otherwise}.
   \end{cases}
\e*
 It follows by a straightforward computation    that, $y\mapsto c_L^R(x,y)$ is concave and ${\li}(c_L^R)\le L_R:= L(R+1)$. In view of Remark 2.6 in Beiglb\"ock et al. \cite{BLO}, there exists an optimizer $(H^*, \vp^*,\ps^*)\in\Dc^{c_L^R}$ for the dual problem $\dd^{c_L^R}(\m,\n)$ such that $\|H^*\|_{\infty}\le 18L_R$ and $\vp^*$, $\ps^*\in\Lam_{19L_R}$, i.e. $
H^*(x)(y-x)+\vp^*(x)+\ps^*(y) \ge c_L^R(x,y)$ for all $(x, y)\in\R^2$ and 
\b*
\dd^{c_L^R}(\m,\n) ~=~ \int\vp^*d\m+\int\ps^*d\n.
\e*
For each $\eps\ge 0$, it follows from the duality $ \pp^{c_L^R}_{\eps}(\m,\n)= \dd^{c_L^R}_{\eps}(\m,\n)$ that 
\b*
&&\hspace*{-0.8cm}\left | \pp_{\eps}^{c_L^R}(\m,\n)-\pp^{c_L^R}(\m,\n) \right |~=~ \left |\dd_{\eps}^{c_L^R}(\m,\n)-\dd^{c_L^R}(\m,\n) \right | ~= ~ \left |\dd_{\eps}^{c_L^R}(\m,\n)-\left[\int \vp^*d\m+\int\ps^* d\n\right] \right |\\
&\le & \left | \left[\int \big(\vp^*+\eps \|H^*\|_{\infty}\big)d\m+\int\ps^* d\n\right]- \left[\int \vp^*d\m+\int\ps^* d\n\right] \right |~=~  \eps\int \|H^*\|_{\infty} d\m ~\le ~  18\eps L_R,
\e*
where the third inequality holds as $\big(H^*, \vp^*+\eps\|H^*\|_{\infty}, \ps^*\big)\in\Dc^{c_L^R}_{\eps}$. In addition, one has by construction $\big|c_L(x,y)-c_L^R(x,y)\big| \le  {\mathds 1}_{(-\infty,-R)\cup(R,\infty)}(y)L\big(|y|-R\big)^2$, which implies that
\b*
\left |\pp^{c_L^R}_{\eps}(\m,\n) - \pp^{c_L}_{\eps}(\m,\n) \right|  ~ \le~  L\int_{(-\infty,-R)\cup(R,\infty)} \big(|y|-R\big)^2\n(dy).
\e*
Therefore,
\b*
&&\left | \pp_{\eps}^c(\m,\n)-\pp^c(\m,\n) \right |~=~  \left | \pp_{\eps}^{c_L}(\m,\n)-\pp^{c_L}(\m,\n) \right |\\
&\le &\left |\pp_{\eps}^{c_L}(\m,\n)-\pp_{\eps}^{c_L^R}(\m,\n)\right | + \left |\pp_{\eps}^{c_L^R}(\m,\n)-\pp^{c_L^R}(\m,\n)\right | + \left |\pp^{c_L^R}(\m,\n)-\pp^{c_L}(\m,\n)\right | \\
&\le &18\eps L(R+1) + 2L\int_{(-\infty,-R)\cup(R,\infty)} \big(|y|-R\big)^2\n(dy),
\e*
which fulfills the proof by setting $\eps=2\eps_n$.
\end{proof}

\begin{remark}
If $\n$ is supported on some closed subset $E\sb\R$, then Theorem \ref{thm:convergence_rate} still holds by assuming that $c$ is Lipschitz on $E^2$ and $\sup_{(x,y)\in E^2}\big|\partial^2_{yy}c(x,y)\big|<\infty$. In addition, it is worth mentioning that the above analysis can be extended to more general functions $c$. Let $c$ be continuous and with linear growth, i.e. $|c(x,y)|\le L(1+|x|+|y|)$ for some $L>0$. Then for every $R\ge 1$, there exists a function $c_R\in\Cc^2(\R^2)$ such that $\sup_{(x,y)\in\bB_R}\big|c(x,y)-c_R(x,y)\big|\le 1/R$, $c_{R}(x,y)=0$ for $(x,y)\notin \bB_{R+1}$ and $\|c_R\|_{\infty} \le \sup_{(x,y)\in \bB_R}|c(x,y)| \le  L(1+2R)$. Further, one has $
\big|c(x,y)-c_R(x,y)\big| 
\le 1/R+ 8L\big(|x|^2+|y|^2\big)/R$, which implies that
 \b*
\big |\pp_{\eps}^{c}(\m,\n)-\pp_{\eps}^{c_R}(\m,\n)\big |~ \le ~ 1/R + 8L\left(\int_{\R} |x|^2\m(dx)+\int_{\R} |y|^2\n(dy)\right)/R ~ =:~ L'/R,
 \e*
Hence, we obtain using the same reasoning $
\big | \pp_{\eps}^c(\m,\n) - \pp^c(\m,\n) \big |
\le\big | \pp_{\eps}^{c_R}(\m,\n)-\pp^{c_R}(\m,\n)\big |+ 2L'/R$. Then $c_R$ satisfies the conditions of Theorem \ref{thm:convergence_rate}. For every $R\ge 1$, using Theorem \ref{thm:convergence_rate}  we deduce a bound on the difference $\big | \pp_{\eps}^{c_R}(\m,\n)-\pp^{c_R}(\m,\n)\big |$. The result can then be optimized over all $R\ge 1$.
\end{remark}

Following the proof of Theorem  \ref{thm:convergence_rate}, we provide below a stability result for the map $\Pc^{\pq} \ni (\m,\n)\mapsto \pp(\m,\n)\in\R$.

\begin{proposition}\label{prop:conti}
Let $\Pc^{\pq}_2\subset \Pc^{\pq}$ be the subset of $(\m,\n)$ with $\n$ having a finite second moment. If $c$ satisfies the conditions of Theorem \ref{thm:convergence_rate}, then there exists $C>0$ such that 
\b*
\big|\pp^c(\m',\n')-\pp^c(\m,\n)\big| ~\le~ C\inf_{R>0} \tilde{\lambda}(R) + \frac{L}{2}\left|\int_{\R}y^2\big(\n'(dy) -\n(dy)\big)\right| \mbox{ for all } (\m,\n),~ (\m',\n')\in \Pc^{\pq}_2, 
\e*
where $\tilde\lambda: (0,\infty)\to \R$ is defined by
\b*
\tilde{\lambda}(R) ~:=~ (R+1) \Wc^{\oplus}\big((\m',\n'),(\m,\n)\big)  + \int_{(-\infty,-R)\cup(R,\infty)} \big(|y|-R\big)^2\big(\n'(dy)+\n(dy)\big).
\e*
For any sequence $\big((\m^n,\n^n)\big)_{n\ge 1}\subset \Pc^{\pq}_2$ satisfying 
 $\lim_{n\to\infty}\Wc^{\oplus}\big((\m^n,\n^n),(\m,\n)\big) = 0$, one has $\lim_{n\to\infty}\pp^c(\m^n,\n^n)= \pp^c(\m,\n)$ if $\lim_{n\to\infty}\int y^2 \n^n(dy)= \int y^2 \n(dy)$.
\end{proposition}

\begin{proof}
Similar to Theorem \ref{thm:convergence_rate}, the key is also the duality. First, one has
\b*
\big|\pp^c(\m',\n')- \pp^c(\m,\n)\big| ~\le~ \big | \pp^{c_L}(\m',\n')- \pp^{c_L}(\m,\n) \big | + \frac{L}{2}\left|\int_{\R}y^2\big(\n'(dy) -\n(dy)\big)\right|
\e*
and $\big | \pp^{c_L}(\m',\n')- \pp^{c_L}(\m,\n) \big |\le \big |\pp^{c_L}(\m',\n')- \pp^{c_L^R}(\m',\n')\big |+\big |\pp^{c_L^R}(\m',\n') -\pp^{c_L^R}(\m,\n)\big |+\big |\pp^{c_L^R}(\m,\n) -\pp^{c_L}(\m,\n)\big |$, 
where $c_L$, $c_L^R:\R^2\to\R$ are defined as same as in the proof of  Theorem \ref{thm:convergence_rate}. Repeating the  arguments in the proof of Theorem \ref{thm:convergence_rate}, it holds that 
\b*
\left |\pp^{c_L}(\m',\n') - \pp^{c_L^R}(\m',\n')\right |&\le & L \int_{(-\infty,-R)\cup(R,\infty)} \big(|y|-R\big)^2\n'(dy), \\
\left |\pp^{c_L}(\m,\n) - \pp^{c_L^R}(\m,\n)\right |&\le & L \int_{(-\infty,-R)\cup(R,\infty)} \big(|y|-R\big)^2\n(dy).
\e*
It remains to estimate $\big |\pp^{c_L^R}(\m',\n') -\pp^{c_L^R}(\m,\n)\big |$. Recall that, in view of Remark 2.6 of \cite{BLO}, $\dd^{c_L^R}(\m',\n')$ is attained by $(H', \vp',\ps')\in\Dc^{c_L^R}$, where $\vp'$, $\ps'\in \Lam_{19L_R}$ with $L_R=L(R+1)$. Therefore,
\b*
&& \hspace*{-0.7cm} \pp^{c_L^R}(\m,\n) - \pp^{c_L^R}(\m',\n') ~=~ \dd^{c_L^R}(\m,\n) -\dd^{c_L^R}(\m',\n') \\
&\le & \left[\int \vp' d\m+\int\ps' d\n\right]- \left[\int \vp' d\m' +\int\ps' d\n'\right] \\
&= & \left[\int \vp' d\m- \int\vp' d\m' \right] + \left[\int \ps' d\n - \int\ps' d\n' \right] 
~\le ~ 19L_R\Wc^{\oplus}\big((\m',\n'),(\m,\n)\big).
\e*
Interchanging $(\m,\n)$ and $(\m',\n')$ and using again the above reasoning, one has $| \pp^{c_L^R}(\m',\n') - \pp^{c_L^R}(\m,\n) | \le 19L_R\Wc^{\oplus}\big((\m',\n'),(\m,\n)\big)$, which concludes the proof.
\end{proof}

We now consider a specific discretization introduced by Dolinsky and Soner \cite{DS1}. We define two sequences of measures  supported on $(k/n)_{k\in\Z}$ as follows:
\begin{equation}
\begin{split}\label{eq:dis_pcoc}
\mu^n\big[\big\{k/n\big\}\big]&~ =~ \int_{[(k-1)/n,(k+1)/n)}\big(1-|nx-k|\big) \mu(dx),  \\
\n^n\big[\big\{k/n\big\}\big]&~ =~ \int_{[(k-1)/n,(k+1)/n)} \big(1-|ny-k|\big) \nu(dy). 
\end{split}
\end{equation}
In the potential theoretic terms of Chacon \cite{Chacon}, $\mu^n$ may be defined as the unique measure supported on $(k/n)_{k\in \Z}$ with its potential function agreeing with that of $\mu$ in those points, i.e.
\b*
\int_{\R} \big|k/n-x\big|\mu(dx) ~=~ \int_{\R} \big|k/n-x\big|\mu^n(dx),~ \mbox{ for all } k\in \Z.
\e*
Then we have the following result.

\begin{proposition}\label{thm:bounded}
\rmi With the notations of \eqref{eq:dis_pcoc}, one has  $(\m^n,\n^n)$, $(\m,\m^n)$, $(\n,\n^n)\in\Pc^{\pq}$ and $\Wc^{\oplus}\big((\m^n, \n^n),(\m, \n)\big)\le 2/n$ for all $n\ge 1$.

\vspace{1mm}

\no \rmii  Let the conditions of Theorem \ref{thm:convergence_rate} hold. Then there exists $C>0$ such that  $\big|\pp^c(\m^n,\n^n)-\pp^c(\m,\n)\big|  \le  C\inf_{R>0} \tilde{\lambda}_n(R)$, where $\tilde{\lambda}_n:(0,\infty) \to\R$ is given by
\b*
\tilde{\lambda}_n(R) ~:=~ \frac{R+1}{n} + \int_{(-\infty,-R)\cup(R,\infty)} \big(|y|-R\big)^2\n(dy).
\e*
\end{proposition}

\begin{proof}
\rmi For any continuous $f:\R\to \R$, define $f^{(n)}: \R\to\R$ by
\b*
f^{(n)}(x) ~:=~ \big(1+\lfloor nx \rfloor -nx\big)f\big(\lfloor nx \rfloor /n\big) + \big(nx-\lfloor nx \rfloor \big)f\big((1+\lfloor nx \rfloor) /n\big),
\e* 
Then it follows from a straightforward computation that $
\int f d\mu^n=\int f^{(n)} d\mu$ and $\int f d\nu^n=\int f^{(n)} d\nu$.
Take $f\equiv 1$, then $f^{(n)}\equiv 1$, and further $\mu^n$ and $\nu^n$ are well defined probability measures. Moreover, taking $f(x)=|x|$, it is clear that $f^{(n)}=f$ and thus $
\int f d\mu^n=\int f^{(n)} d\mu<\infty$ and $\int f d\n^n=\int f^{(n)} d\n<\infty$. 
To prove $(\m^n,\n^n)$, $(\m,\m^n)$, $(\n,\n^n)\in\Pc^{\pq}$, it suffices to test for $f(x)=(x-K)^+$. It follows easily that $f^{(n)}$ is convex and $f^{(n)}\ge f$ by computation. This implies that $(\m^n,\n^n)$, $(\m,\m^n)$, $(\n,\n^n)\in\Pc^{\pq}$. To end the proof, we notice that $|\int f d\mu^n -\int f d\mu| \le \int |f^{(n)}-f|d\mu \le 1/n$, which yields $\Wc(\mu^n,\mu)\le 1/n$ by \eqref{def:wass_dual}. 

\vspace{1mm}

\no \rmii It suffices to apply Proposition \ref{prop:conti} with $\m':=\m^n$ and $\n':=\n^n$. Using the construction of $\n^n$, one has
\b*
\int_{(-\infty,-R)\cup(R,\infty)} \big(|y|-R\big)^2\n^n(dy) &\le& \int_{(-\infty,-R)\cup(R,\infty)} \big(|y|-R\big)^2\n(dy) + \frac{1}{n^2}, \\
\left|\int_{\R}y^2\big(\n^n(dy) -\n(dy)\big)\right| &\le &  \frac{1}{4n^2},
\e*
which concludes the proof.
\end{proof}

\section{Summary and possible extensions}
\label{sec:extension}
We believe that our paper offers an important and pioneering contribution to computational methods for \textsf{MOT} problems. Our first main result, Theorem \ref{thm:general}, establishes an approximation result of a general  \textsf{MOT} problem $\pp(\bm)$ via \textsf{LP} problems by discretizing the marginal distributions and relaxing the martingale condition. Further, we introduce two kinds of approximations: a deterministic one, $\bm^n$, and a stochastic one, $\hbm$. We investigate $\Wc^{\oplus}\big(\bm^n,\bm\big)$ and $\E\big[\Wc^{\oplus}\big(\hbm,\bm\big)\big]$ such that  the computation of suitable \textsf{LP} problems $\pp_{\eps_n}(\bm^n)$ and $\pp_{\eps_m}(\bm^{\hnm})$ can be carried out. In addition, we provide some numerical examples for illustration. 

Our second main result, Theorem \ref{thm:convergence_rate},  provides an estimation on the convergence rate for the one-dimensional case.  This result, in particular, allows us to deduce a complete scheme for calculating $\pp(\m,\n)$ to a given precision.

As a relatively immediate, but practically relevant, extension of our setup, for the computation of $\pp(\vec{\bm})$ defined in \eqref{def:mmot}, Theorem \ref{thm:general} can be easily extended to show $\lim_{n\to\infty} \pp_{\eps_n}(\vec{\bm}^n)= \pp(\vec{\bm})$, where the sequence $(\eps_n)_{n\ge 1}$ converges to zero and satisfies 
\b*
\eps_n ~\ge ~ \sum_{k=1}^N\sum_{i=1}^d\Wc\big(\m^n_{k,i},\m_{k,i}\big),~ \mbox{ for all } n\ge 1.
\e*
Further investigation of this setup, which is of practical relevance, is an ongoing work.

Last, but not least, we point out that solving efficiently  the \textsf{LP} problems $\pp_{\eps_n}(\bm^n)$ and $\pp_{\eps_m}(\bm^{\hnm})$ is also an interesting avenue of research and may attract the attention from practitioners. We notice that $\pp_{\eps_n}(\bm^n)$ and $\pp_{\eps_m}(\bm^{\hnm})$ are in fact  \textsf{LP} problems with a particular structure, i.e. the constraints are given by a sparse matrix, and some existing algorithms can be extended to their setup: 

\begin{itemize}
\item If $N=2$ and $d=1$, the iterative Bregman projection in \cite{BCCNP} can be applied to solve $\pp_{\eps_n}(\m^n,\n^n)$ with an additional entropic regularization.
\item If $N=2$, the stochastic averaged gradient approach, see e.g. Genevay et al. \cite{GCPB}, may deal with $\pp_{\eps_n}(\m^n,\n^n)=\dd_{\eps_n}(\m^n,\n^n)$ by the duality. 
\end{itemize}  

We believe that extending the above algorithms  to multiple steps and  higher dimensions is an important and challenging problem. 

\appendix
\section{Supplementary proofs}
\label{sec:append}

\begin{proof}[Proof of Theorem \ref{thm:duality}]
The existence of $\P^*$ is a consequence the compactness of $\Mc_{\eps}(\bm)$. As for the duality, 
we prove a slightly stronger result. Let $\overline \Hc\subset \Hc$ be the subset of $H=(H_k)_{1\le k\le N-1}$ such that $H_k\in \Cc_b(\Om^k;\R^d)$ for $k=1,\ldots, N-1$. Define the minimization problem:
\b*
\overline \dd_{\eps}(\bm) ~:=~ \inf_{(H,\ps)\in\overline\Dc_{\eps}}~ \left[\sum_{k=1}^N\int_{\R^d} \ps_k(\bx) \m_k(d\bx)\right],~ \mbox{ where } \overline \Dc_{\eps}:=\Dc_{\eps} \cap \big(\overline \Hc\times \Lambda^N\big).
\e*
Then, by definition, $\pp_{\eps}(\bm)\le \dd_{\eps}(\bm)\le \overline \dd_{\eps}(\bm)$. Define the function $\Phi: \Pc(\bm)\times \overline\Hc \to \R$ by
\b*
\Phi(\P, H) ~:=~ \E_{\P}\left[c(\bS_1,\ldots, \bS_N) -\sum_{k=1}^{N-1} H_k(\bS_1,\ldots, \bS_k)\cdot (\bS_{k+1}-\bS_k)\right] +  \eps\sum_{k=1}^{N-1} \|H_k\|_{\infty}.
\e*
Since $\Phi(\cdot,H)$ is continuous and concave for all $H\in\overline\Hc$ and $\Phi(\P,\cdot)$ is continuous and convex for all $\P\in\Pc(\bm)$, then it holds that $
\sup_{\P\in\Pc(\bm)} \inf_{H\in\overline \Hc} \Phi(\P,H) =  \inf_{H\in\overline \Hc} \sup_{\P\in\Pc(\bm)}  \Phi(\P,H)$ in view of the Min-Max theorem as $\Pc(\bm)$ is convex and compact. Hence, 
\b*
&&\overline \dd_{\eps}(\bm) \\
&=& \inf_{H\in\overline\Hc}~ \inf_{\psi\in\Lambda^N:~  \sum_{k=1}^N \psi(\bx_k) \ge c(\bx_1,\ldots, \bx_N) - \sum_{k=1}^{N-1}\big(H_k(\bx_1,\ldots \bx_k)\cdot (\bx_{k+1}-\bx_k)-\eps\|H_k\|_{\infty}\big)}~ \sum_{k=1}^N \int_{\R^d} \ps_k(\bx) \m_k(d\bx)  \\
&=&  \inf_{H\in\overline\Hc} \sup_{\P\in\Pc(\bm)}~ \Phi(\P, H) ~=~ \sup_{\P\in\Pc(\bm)} \inf_{H\in\overline \Hc}~ \Phi(\P, H) \\
&=& \sup_{\P\in\Mc_{\eps}(\bm)} \inf_{H\in\overline \Hc}   ~\Phi(\P, H) ~\le~ \sup_{\P\in\Mc_{\eps}(\bm)}~ \E_{\P}\big[c(\bS_1,\ldots, \bS_N)\big] ~=~ \pp_{\eps}(\bm),
\e*
where the second equality follows from the classical duality of Kantorovich, and the fourth equality is by the fact $\inf_{H\in\overline \Hc} \Phi(\P, H) =-\infty$ once $\P\notin\Mc_{\eps}(\bm)$.
\end{proof}

\begin{proof}[Proof of Proposition \ref{prop:sto_disc}]
As $\Wc^{\oplus}\big(\hbm,\bm\big) =\sum_{k=1}^N \Wc\big(\hm_k,\m_k\big)$ with every $\m_k$ satisfying Assumption \ref{ass:mom}, it suffices to deal with $\Wc(\hm_1,\m_1)$. For notational simplicity we write $\hm\equiv \hm_1$ and $\m\equiv \m_1$. In the rest of the proof, we refer to \cite{FG}. Combining Lemmas 5 and 6 together with the proof of Theorem 1, it holds that 
\b*
\E\big[\Wc(\hm,\m)\big] 
~\le ~ 24(M_{\theta}+1)d^{(1-\theta)/2}2^{\theta} \sum_{i\ge 0} 2^{i} \sum_{j\ge 0} 2^{-j} \min \left( \eps_i, 2^{dj/2} (\eps_i/n)^{1/2} \right),~ \mbox{ with } \eps_i := 2^{-\theta i}.
\e*
For every $\eps\in (0,1)$, it follows by a straightforward computation that
\b*
\sum_{j\ge 0} 2^{-j} \min \left( \eps, 2^{dj/2} (\eps/n)^{1/2} \right)
  ~\le ~ \frac{9}{2\log 2} \begin{cases}
\min\left(\eps, (\eps/n)^{1/2}\right) &   \mbox{if }  d=1, \\
\min\left(\eps, (\eps/n)^{1/2}\log(2+\eps n)\right)  &   \mbox{if }  d=2, \\ 
\min\left(\eps, \eps(\eps n)^{-1/d}\right)  &  \mbox{if }  d>2.
\end{cases}
\e*
Next we substitute $\eps$ by $\eps_i$ and distinguish different $d$. If $d=1$, then
\b*
\sum_{i\ge 0} 2^{i} \min\left(\eps_i, (\eps_i/n)^{1/2}\right) ~\le~  \begin{cases}
2\sqrt{2}n^{1/\theta-1} \big/\big((2^{1-\theta/2}-1)(1-2^{1-\theta})\big) &  \mbox{if }  \theta<2, \\
4(1+\log n) n^{-1/2}  &   \mbox{if }  \theta=2, \\ 
n^{-1/2} \big/ \big(1-2^{1-\theta/2}\big) &  \mbox{if }  \theta>2.
\end{cases}
\e*
If $d=2$, then
\b*
\sum_{i\ge 0} 2^{i} \min\left(\eps_i, (\eps_i/n)^{1/2}\log(2+\eps_i n)\right)  ~\le~   \begin{cases}
7n^{1/\theta-1} \big/ (2^{1-\theta/2}-1)^2 &   \mbox{if }  \theta<2, \\
6\big(1+(\log n)^2\big) n^{-1/2}  & \mbox{if }  \theta=2, \\ 
(1+\log n)n^{-1/2} \big/ \big(1-2^{1-\theta/2}\big)  & \mbox{if }  \theta>2.
\end{cases}
\e*
If $d>2$, then
\b*
\sum_{i\ge 0} 2^{i} \min\left(\eps_i, \eps_i(\eps_i n)^{-1/d}\right)   ~\le~   \begin{cases}
3n^{1/\theta -1}\big/ \big((2^{1-\theta(1-1/d)}-1)(1-2^{1-\theta})\big)  &  \mbox{if }  \theta<d/(d-1), \\
6(1+\log n) n^{-1/d}  &   \mbox{if }  \theta=d/(d-1), \\ 
n^{-1/d} \big / \big(1-2^{1-\theta(1-1/d)}\big)   &  \mbox{if }  \theta>d/(d-1).
\end{cases}
\e*
The proof is completed with $C(\theta, d)$ being the product of the corresponding coefficients.
\end{proof}

\begin{lemma}\label{lem:Skorokhod}
With the same conditions and notations of Proposition  \ref{prop:approx}, there exist measurable functions $f_1,\ldots, f_N: \Om^2\to\R^d$ such that $\Qc \circ (X_k, Y_k)^{-1} = \P_k$, where $ Y_k:= f_k(X_k, Z_k)
$ for $k=1,\ldots, N$.
\end{lemma}

\begin{proof}
Without loss of generality, we only prove for $k=1$. Further, we drop the subscript without any danger of confusion, i.e. $X\equiv X_1$, $Z\equiv Z_1$, $\m\equiv \m_1$, $\P\equiv \P_1$, etc. Disintegrating with respect to $\m$, one has $\P(d\bx,d\by)=\m(d\bx)\otimes \lambda_{\bx}(d\by)$, where $\big(\lambda_{\bx}(d\by)\big)_{\bx\in\R^d}$ denotes the regular conditional probability distribution (r.c.p.d.). Hence, the above  claim is equivalent to the existence of a measurable function $f: \Om^2\to\R^d$ satisfying, for $\m$ - a.e. $\bx\in\R^d$
\b*
\Qc\big[f(\bx,Z)\in A\big| X=\bx\big]~=~ \lambda_{\bx}(A),~ \mbox{ for all } A\in\Bc(\R^d),
\e*
or namely, $f(\bx,\cdot)$ transfers the law of $Z$ to $\lambda_{\bx}$ for $\m$ - a.e. $\bx\in\R^d$. We first prove this claim for the case of $d=1$, i.e. $\bx=x$ and then conclude for the general case.

\vspace{1mm}

\no \rmi Let $F$ and $G_x$ be respectively the cumulative distribution functions of $Z$ and $\lambda_x$, and define the right-continuous inverse by $G_x^{-1}(t) := \inf \big\{y\in\R: G_x(y)>t \big\}$. Define further $f:\R^2\to\R$ by $f(x,y):=G_x^{-1}\circ F(y)$, then $f$ is  measurable by the definition of r.c.p.d. Moreover, it follows by Villani \cite{Villani}, Pages 19-20, that for $\m$ - a.e. $x\in\R$ one has $\Qc\big[Y\in A\big| X=x\big]= \lambda_x(A)$ for all $A\in\Bc(\R)$, 
which concludes the proof.

\vspace{1mm}

\no \rmii Now let us treat the general case. Recall that $\bx=(x_1,\ldots, x_d)$,  $\by=(y_1,\ldots, y_d)$ and $\bz=(z_1,\ldots, z_d)$. We proceed as follows.

\vspace{1mm}

\emph{Step 1}: Take the  marginal distributions on the first coordinate for $Z$ and $\lambda_{\bx}$, denoted by $F_1(z_1)dz_1$ and $\lambda_{\bx}^1(dy_1)$, where we note that $Z$ admits a density function on $\R^d$. Then repeat the the procedure of {\rm (i)}, and construct the measurable  map $f_1(\bx,\cdot)$ which may transfers the law $F_1(z_1)dz_1$ to the other one $\lambda_{\bx}^1(dy_1)$.

\vspace{1mm}

\emph{Step 2}: Next take the marginals on the first two coordinates for $Z$ and $\lambda_{\bx}$, $F_2(z_1,z_2)dz_1dz_2$ and $\lambda_{\bx}^2(dy_1,dy_2)$, and disintegrate them with respect to the first one. This yields 
\b*
F_2(z_1,z_2)dz_1dz_2~:=~F_1(z_1)dz_1\otimes F_{z_1,2}(z_2)dz_2 &\mbox{and} & \lambda_{\bx}^2(dy_1, dy_2)~:=~\lambda_{\bx}^1(dy_1)\otimes \lambda^2_{\bx,y_1}(dy_2).
\e*
For each $z_1$, set $y_1=f_1(\bx,z_1)$, and define $f_{1,2}(\bx, z_1, \cdot)$ according to {\rm (i)},  which transfers thus $F_{z_1,2}(z_2)dz_2$ to $\lambda^2_{\bx,f_1(\bx,z_1)}(dy_2)$.

\vspace{1mm}

\emph{Step 3}: We repeat the construction of \emph{Step 2} by adding coordinates one after the other and defining $f_{1,2,3}(\bx,z_1,z_2, \cdot)$,  etc. After $N$ steps, this produces the required map $f(\bx,z)$ which transports the law of $Z$ to $\lambda_{\bx}$. 
\end{proof}



\end{document}